\documentclass[11pt,reqno]{amsart}
\usepackage{amsfonts,amssymb,amsmath,amsopn,amsthm,graphicx}
\usepackage{amsxtra, mathrsfs}
\usepackage{colordvi}
\usepackage[usenames,dvipsnames]{color}
\usepackage{amsfonts,amssymb,amsbsy,amsmath,amsthm,dsfont}

\usepackage{amsthm}
\usepackage{amsbsy}
\usepackage{amstext}
\usepackage{amsfonts}
\usepackage{mathcomp}
\usepackage{mathtools}
\usepackage{dsfont}
\usepackage{tabularx}
\usepackage{latexsym}
\usepackage{amssymb}
\usepackage{esint}
\usepackage[colorlinks,pdfpagelabels,pdfstartview = FitH,bookmarksopen = true,bookmarksnumbered = true,linkcolor = blue,
plainpages = false,hypertexnames = false,citecolor = red,pagebackref=false]{hyperref}
\usepackage{cite}

\usepackage{lipsum}
\usepackage[many]{tcolorbox}
\usetikzlibrary{decorations.pathreplacing}
 \numberwithin{equation}{section}

\newtcolorbox{leftbrace}{ %
	enhanced jigsaw, 
	breakable, 
	frame hidden, 
	overlay={%
		\draw [
		decoration={brace,amplitude=0.5em},
		decorate,
		ultra thick,
		]
		(frame.south west)--(frame.north west);
	},
	parbox=false,
}

\newtheorem{theorem}{Theorem}[section]

\newtheorem{cor}[theorem]{Corollary}

\newtheorem{lem}[theorem]{Lemma}
\newtheorem{conjecture}[theorem]{Conjecture}
\newtheorem{lemma}[theorem]{Lemma}

\newtheorem{example}[theorem]{Example}
\newtheorem{assumptions}[theorem]{Assumptions}
\newtheorem{definition}[theorem]{Definition}

\newtheorem{remark}[theorem]{Remark}

\newtheorem{rem}[theorem]{Remark}
\newtheorem{theo}[theorem]{Theorem}
\newtheorem{prop}[theorem]{Proposition}
\newtheorem{remarks}[theorem]{Remarks}

\newtheorem{defin}[theorem]{Definition}

\renewcommand{\div}{\operatorname{div}}

\newcommand{\N}{\ensuremath{\mathbb{N}}}
\newcommand{\R}{\ensuremath{\mathbb{R}}}

\newcommand{\C}{\mathrm{C}}

\newcommand{\W}{\mathrm{W}}

\def\Xint#1{\mathchoice
    {\XXint\displaystyle\textstyle{#1}}%
    {\XXint\textstyle\scriptstyle{#1}}%
    {\XXint\scriptstyle\scriptscriptstyle{#1}}%
    {\XXint\scriptscriptstyle\scriptscriptstyle{#1}}%
    \!\int}
\def\XXint#1#2#3{\setbox0=\hbox{$#1{#2#3}{\int}$}
    \vcenter{\hbox{$#2#3$}}\kern-0.5\wd0}

\def\dashint{\Xint{\raise4pt\hbox to7pt{\hrulefill}}}

\def\XXiint#1#2#3{\setbox0=\hbox{$#1{#2#3}{\iint}$}
    \vcenter{\hbox{$#2#3$}}\kern-0.5\wd0}

\renewcommand{\epsilon}{\varepsilon}

\renewcommand{\rho}{\varrho}

\renewcommand{\epsilon}{\varepsilon}
\renewcommand{\rho}{\varrho}
\renewcommand{\d}{\:\! \mathrm{d}}

\DeclareMathOperator{\loc}{loc}

\numberwithin{equation}{section}
\allowdisplaybreaks
\newcommand{\diver}{\mathrm{div}\,}
\newcommand{\RN}[1]{%
	\textup{\uppercase\expandafter{\romannumeral#1}}%
}

\newcommand{\pd}{\partial}

\newcommand{\dx}{\,\mathrm{d}x}

\newcommand{\dS}{\,\mathrm{d}S}

\newcommand{\be}{\begin{equation}}
\newcommand{\ee}{\end{equation}}
\newcommand{\bea}{\begin{eqnarray}}
\newcommand{\eea}{\end{eqnarray}}
\newcommand{\bean}{\begin{eqnarray*}}
	\newcommand{\eean}{\end{eqnarray*}}

\newlength{\wex}  \newlength{\hex}
\newcommand{\ass}[1]{Let Assumptions~\ref{assump1} hold  in a bounded Lipschitz domain $\Gw$}

\def\ga{\alpha}            
       \def\gd{\delta}      
                         \def\vge{\varepsilon}
\def\gf{\phi}       \def\vgf{\varphi}    \def\gh{\eta}
            \def\gl{\lambda}
\def\gm{\mu}                 
    \def\gr{\rho}        
       
      \def\gw{\omega}
                \def\gz{\zeta}

\def\Gw{\Omega}              

\begin{document}
\renewcommand{\refname}{References} 
\renewcommand{\abstractname}{Abstract} 
\title[Fuchsian potential in Morrey space]{Positive Liouville theorem and asymptotic behaviour for $(p,A)$-Laplacian type elliptic equations with Fuchsian potentials in Morrey space}
\author[R.Kr. Giri]{Ratan Kr. Giri}
\address{Ratan Kr. Giri\\
	Department of Mathematics, Technion - Israel Institute of Technology\\
	3200 Haifa, Israel}
\email{giri@campus.technion.ac.il/giri90ratan@gmail.com}
\author[Y.~Pinchover]{Yehuda Pinchover}
\address{Yehuda Pinchover\\
	Department of Mathematics, Technion - Israel Institute of Technology\\
	3200 Haifa, Israel}
\email{pincho@technion.ac.il}
\dedicatory{Dedicated to Volodya Maz'ya on the occasion of his 80th birthday}	
\begin{abstract}
We study Liouville-type theorems and the asymptotic behaviour of positive solutions near an isolated singular point $\zeta\in\partial\Omega\cup\{\infty\}$ of the quasilinear elliptic equations 
$$-\text{div}(|\nabla u|_A^{p-2}A\nabla u)+V|u|^{p-2}u =0\quad\text{in } \Omega\setminus\{\zeta\},$$
 where $\Omega$ is a domain in $\mathbb{R}^d$ ($d\geq 2$), and $A=(a_{ij})\in L_{\rm loc}^{\infty}(\Gw;\R^{d\times d})$ is a symmetric and locally uniformly positive definite matrix. The potential $V$ lies in a certain local Morrey space (depending on $p$) and has a Fuchsian-type isolated singularity at $\zeta$.
\end{abstract}
%
\maketitle
\section{Introduction}
Let $\Omega$ be a domain in $\mathbb{R}^d$, $d\geq 2$, and consider the quasilinear elliptic partial differential equation
	\vspace{-3mm} 
\begin{equation}\label{p-laplace equ_1}
Q(u)=Q_{p,A,V}(u) := -\Delta_{p,A}(u)+V|u|^{p-2}u =0\qquad \text{in } \Omega.
\end{equation}
and let $\gz\in \{0,\infty\}$ be a fixed isolated singular point of $Q_{p,A,V}$ which belongs to the ideal boundary of $\Omega$ (to be explained in the sequel).

Here  $1<p<\infty$, $V$ is a real valued potential belonging to a certain local Morrey space, and  
$$\Delta_{p,A}(u):=\text{div}(|\nabla u|_A^{p-2} A\nabla u)$$ is the $(p,A)$-Laplacian, where $A=(a_{ij})\in L_{\rm loc}^{\infty}(\Gw;\R^{d\times d})$ is a symmetric and locally uniformly positive definite matrix, and
\be\nonumber
|\xi|_{A(x)}^2
:=
A(x)\xi\cdot\xi=\sum_{i,j=1}^d a_{ij}(x)\xi_i\xi_j
\qquad x\in\Gw\mbox{ and }\xi=(\xi_1,...,\xi_d)\in\R^d.
\ee
We note that \eqref{p-laplace equ_1} is the Euler-Lagrange equation associated to the energy functional 
\begin{equation}\label{eqn1}
\mathcal{Q}(\vgf)=\mathcal{Q}_{p,A,V}(\vgf):=\int_\Omega (|\nabla \vgf|^p_A+V|\vgf|^p)\d x \qquad\vgf\in C_c^\infty(\Omega).
\end{equation}

The quasilinear equation \eqref{p-laplace equ_1} satisfies the homogeneity property of linear equations but not the additivity (therefore, such an equation is sometimes called {\em half-linear} or quasilinear elliptic equations with {\em natural growth terms}). Consequently, one expects that positive solutions of 
\eqref{p-laplace equ_1} would share some properties of positive solutions of linear elliptic equations. Indeed, criticality theory for \eqref{p-laplace equ_1}, similar to the linear case, was established in \cite{PT,PR,pinchover_psaradakis}.

  In \cite{frass_pinchover}, Frass and Pinchover studied Liouville theorems and removable singularity theorems for positive {\em classical} solutions of \eqref{p-laplace equ_1} under the assumptions that $A$ is the {\em identity matrix}, $V\in L^\infty_{\loc}(\Omega)$, and $V$ has a {\em pointwise} Fuchsian-type singularity near $\gz\in \{0,\infty\}$, namely, 
\begin{equation}\label{eq_pFs}
|V(x)|\leq \frac{C}{|x|^p}\qquad \mbox{near } \gz.
\end{equation} 
Moreover, in the same paper and  in \cite{frass_pinchover2}, the asymptotic behavior of the {\it quotient} of two positive solutions near the singular point $\zeta$ has been obtained. The results in \cite{frass_pinchover,frass_pinchover2} extend the results obtained in \cite[and the references therein]{pinchover} for second-order {\em linear} elliptic operators (not necessarily symmetric)  to the quasilinear case. We note that an affirmative answer to Problem~51 of Maz'ya's recent paper \cite{maz} follows from  \cite[Theorem~1.1]{frass_pinchover}. 

\medskip

The aim of the present paper is to study Liouville-type theorems, Picard-type principles, and removable singularity theorems for positive {\em weak} solutions of \eqref{p-laplace equ_1}, by relaxing significantly the condition on the potential $V\in L^\infty_{\loc}(\Omega)$. More precisely, we enable a symmetric, locally bounded, and locally uniformly positive definite matrix $A$, and a potential $V$ that lies in a certain local Morrey space and has a generalized  Fuchsian-type singularity  at $\gz$ (in term of a weighted Morrey norm of $V$). In fact, our local regularity assumptions on $A$ and $V$ are almost the weakest to keep the validity of the local Harnack inequality and the local H\"older continuity of weak solutions. 

\medskip

The outline of the present paper is as follows. In Section~\ref{sec_pre} we provide a short summary on the local theory of positive solutions of \eqref{p-laplace equ_1} with potentials in local Morrey spaces and prove Harnack convergence principle under minimal assumptions on the  sort of convergence of the coefficients of the sequence of operators.  In Section~\ref{sec_fuchs} we introduce the notion of a (generalized) Fuchsian singularity for the operator $Q$ at a point $\gz$, and prove a uniform Harnack inequality near such a singular point which is a key result for proving (under further assumptions) that  the quotient of two positive solutions near $\gz$ admits a limit in the wide sense. Section~\ref{sec_pA_harmonic} is devoted to the asymptotic behaviour of positive $(p,\mathbb A)$-harmonic functions near an isolated singular point for the case where $\mathbb A\in \R^{d\times d}$ is a symmetric and positive definite matrix. In Section~\ref{sec_weak_Fuchs} we assume that $Q$ has a {\em weak} Fuchsian singularity at $\gz$ and prove that it is a sufficient condition for the validity of a positive Liouville-type theorem. Finally, Section~\ref{sec_spher} is devoted to the  study of Liouville-type theorem in the elliptically symmetric case.

\section{Preliminaries}\label{sec_pre}
We begin with notation, some definitions and assumptions. Throughout the paper, $\Omega$ is a domain (i.~e., a nonempty open connected set) in $\mathbb{R}^d$, $d\geq 2$. By $B_r(x_0)$ and $S_r(x_0)= \partial B_r(x_0)$, we denote the open ball and the sphere of radius $r>0$ centered at $x_0$, respectively, and we set $B_r:=B_r(0), S_r:=S_r(0)$.  Denote $B_r^*:=\mathbb{R}^d\setminus\overline{B_r}$ and $(\R^d)^*:=\R^d\setminus \{0\}$, the corresponding exterior domains. For $R>0$ we denote by $\mathcal{A}_R$  the annuls
$\mathcal{A}_R:=\{x\in\mathbb{R}^d \mid  {R}/{2}\leq |x|< {3R}/{2} \},$  
and for a domain $\Omega\subset \mathbb{R}^d$ and $R>0$, we define the dilated domain 
$\Omega/R:= \{x\in \mathbb{R}^d\,|\, x=R^{-1}y,\, \text{where}\,y\in\Omega\}.$
Let $f,g\in C(\Omega)$ be two positive functions. The notation $f\asymp g$ in $\Omega$ means that there exists positive constant $C$ such that 
$$C^{-1}g(x)\leq f(x)\leq C g(x)\quad \text{ for all } \,x\in \Omega.$$
We write $\Omega_1\Subset\Omega_2$ if $\Omega_2$ is open and $\overline{\Omega}_1$ is compact (proper) subset of $\Omega_2$. By a {\em compact exhaustion} of a domain $\Gw$,  we mean a sequence of smooth relatively
compact domains in $\Gw$ such that $\Gw_1 \neq \emptyset$,  $\Gw_i \Subset \Gw_{i + 1}$, and 
$\cup_{i = 1}^{\infty} \Gw_i = \Gw$.  
Finally, throughout the paper $C$ refers to a positive constant which may vary from  line to line.

\medskip

We first introduce a certain class of Morrey spaces, in which the potential $V$ of the operator $Q_{p,A,V}$ belongs to.
\begin{defin}[Morrey spaces]{\em 
	A function $f\in L^1_{\loc}(\Omega;\R)$ is said to belong to the local Morrey space $M^q_{\loc}(\Omega;\R)$, $q\in[1, \infty]$ if for any $\omega\Subset\Omega$
	$$\|f\|_{M^q(\omega)}:= \sup_{\substack{y\in \omega\\0<r<\text{diam}(\omega)}}\frac{1}{r^{d/q'}}\int_{\omega\cap B_r(y)}|f|\dx<\infty,$$
	where $q'=q/(q-1)$ is the H\"older conjugate exponent of $q$. By applying H\"older inequality, it can be seen that $L^q_{\loc}(\Omega)\subsetneq M^q_{\loc}(\Omega)\subsetneq L^1_{\loc}(\Omega)$ for any $q\in(1, \infty)$. For $q=1$ we have $M^1_{\loc}(\Omega)=L^1_{\loc}(\Omega)$, and for $q=\infty$  we have $M^\infty_{\loc}(\Omega)= L^\infty_{\loc}(\Omega)$ (as vector spaces).
}
\end{defin}
Next we define a special local Morrey space $M^q_{\loc}(p;\Omega)$ which depends on the underlying exponent $1<p<\infty$. 
\begin{defin}[Special Morrey spaces]\label{morrey_space}{\em 
	For $p\neq d$, we define
	$$M^q_{\loc}(p;\Omega):=\begin{cases}
	M^q_{\loc}(\Omega) \mbox{ with } q>d/p &\mbox{ if } p<d,\\
	L^1_{\loc}(\Omega)  & \mbox{ if } p>d,
	\end{cases}$$
	while for $p=d$, the Morrey space $M^q_{\loc}(d;\Omega)$ consists of all those $f$ such that for some $q>d$ and any $\omega\Subset\Omega$
		\vspace{-3mm} 
	$$ \|f\|_{M^q(d;\omega)}:= \sup_{\substack{y\in \omega\\0<r<\text{diam}(\omega)}}\varphi_q(r)\int_{\omega\cap B_r(y)}|f|\dx<\infty,$$
	where $\varphi_q(r):=\log ^{q/d'}\left(\frac{\text{diam}(\omega)}{r}\right)$ (see \cite[Theorem~1.94]{maly_ziemer}, and references therein).
}
\end{defin}
Throughout the article we consider the following $(p,A)$-Laplace type equation 
\begin{equation}\label{p_laplace equ_2}
Q(u)=Q_{p,A,V}(u) := -\Delta_{p,A}(u)+V|u|^{p-2}u =0\qquad \text{in } \Omega,
\end{equation}
under the following assumptions on $A$ and $V$.
\begin{assumptions} \label{assump} 
		\begin{leftbrace}
		\begin{itemize}
			\item[{\ }]		
			\item $A = (a_{ij})_{i,j=1}^{d} \in L_{\rm loc}^{\infty}(\Gw;\R^{d\times d})$ is a symmetric matrix.
			 \item $A$ is locally uniformly elliptic  in $\Gw$, that is, for any compact $K\subset \Gw$ there exists  $\Theta_K>0$ such that 
			\begin{eqnarray*} 
			\hspace*{2.5cm}\Theta_K^{-1}\sum_{i=1}^d\xi_i^2\leq\sum_{i,j=1}^d
			a_{ij}(x)\xi_i\xi_j\leq \Theta_K\sum_{i=1}^d\xi_i^2 \quad \forall \xi\in \mathbb{R}^n \mbox{ and } \forall x\in K.
			\end{eqnarray*}
			\item  $V\in M_{\loc}^q(p;\Gw)$ is a real valued function.
		\end{itemize}
	\end{leftbrace}
\end{assumptions}
Recall that a function $v$ is said to be a {\em (weak) solution} of the equation $Q_{p,A,V}(u)=0$ in $\Gw$  if  $v\in W_{\loc}^{1,p}(\Omega)$ and $v$ satisfies  
\begin{equation}\label{eq-ws}
\int_\Omega(\nabla v|^{p-2}_A A\nabla v\cdot \nabla \varphi+ V|v|^{p-2}v\varphi)\d x=0\qquad \forall \varphi\in C_c^\infty(\Omega).
\end{equation}
Further, we say the $v\in W_{\loc}^{1,p}(\Omega)$ is a {\em supersolution} of 
\eqref{p_laplace equ_2} if the integral in \eqref{eq-ws} is nonnegative for every nonnegative test function $\varphi\in \C_c^\infty(\Omega)$. A function $v$ is a {\em subsolution} of  \eqref{p_laplace equ_2} if $-v$ is supersolution of \eqref{p_laplace equ_2}. It should be noted that the above definitions make sense because of the following Morrey-Adams Theorem (see for example,  \cite[Theorem 2.4]{pinchover_psaradakis} and references therein).
\begin{theorem}[Morrey-Adams theorem]\label{theorem:uncertainty}
	Let $\gw\Subset\R^d$ and $V\in M^q(p;\gw)$.
	
	$(i)$ There exists a constant $C=C(d,p,q)>0$ such that for any $\delta>0$ and all $u\in\W_0^{1,p}(\gw)$
	\be\label{ineq:uncertainty}
	\int_\gw|V||u|^p\dx
	\leq
	\delta\|\nabla u\|^p_{L^p(\gw;\R^d)}
	+
	\frac{C}{\delta^{d/(pq-d)}}\|V\|_{M^q(p;\gw)}^{pq/(pq-d)}\|u\|^p_{L^p(\gw)}.
	\ee
	
	$(ii)$ For any $\gw'\Subset\gw$ with Lipschitz boundary, there exist $0<C=C(d,p,q,\gw',\gw,\gd,\|V\|_{M^q(p;\gw)})$ and $\delta_0$ such that for $0<\delta\leq \delta_0$ and all $u\in W^{1,p}(\gw')$
	\be\nonumber
	\int_{\gw'}|V||u|^p\dx
	\leq
	\delta\|\nabla u\|^p_{L^p(\gw';\R^d)}
	+
	C\|u\|^p_{L^p(\gw')}.
	\ee
\end{theorem}

We recall the Allegretto-Piepenbrink-type  theorem (see,  \cite[Theorem~4.3]{pinchover_psaradakis}). This theorem  states that  $\mathcal{Q}_{p,A,V}(\vgf)\geq 0$ for all $\vgf\in C_c^\infty(\Omega)$ (in short, $\mathcal{Q}_{p,A,V}\geq 0$ in $\Gw$) if and only if the equation $Q_{p,A,V}(u) = 0$ possesses a positive (super)solution in $\Gw$. 


{\bf Throughout the paper, we assume that $\mathcal{Q}_{p,A,V}(\vgf)\geq 0$ for all $\vgf\in C_c^\infty(\Omega)$}.


The above assumption implies the solvability of the Dirichlet problem in bounded subdomains (see \cite{pinchover_psaradakis}  Theorem~3.10 and Proposition~5.2):
   
\begin{lemma}\label{lem_Dirichlet}
	Assume that $\mathcal{Q}_{p,A,V}\! \geq \! 0$ in $\Gw$. 
	Then for any Lipschitz subdomain $\gw \Subset \Gw$, $0\leq  g \in  C(\gw)$ and $0\! \leq \!f \!\in \! C(\partial \gw)$, there exists a nonnegative solution $u\in W^{1,p}(\gw)$ of the problem 
	$$Q_{p,A,V}(v)=g  \mbox{ in } \gw,  \mbox{ and } v=f \mbox{ on } \partial \gw.$$
  Moreover,  the solution $u$ is unique if either $f=0$ or $f>0$ on $\partial \gw$. 
\end{lemma}
We recall the local Harnack inequality of nonnegative solutions of \eqref{p_laplace equ_2}, see for example, \cite [Theorem~3.14]{maly_ziemer} for the case $p\leq d$ and  \cite[Theoren 7.4.1]{pucci_serin} for the case $p>d$.
\begin{theo}[Local Harnack inequality]\label{harnack_ineq}
	Let   $A, V,$ satisfy Assumptions~\ref{assump}, and let $\omega'\Subset\omega\Subset\Omega$. Then for any nonnegative solution $v$  of \eqref{p_laplace equ_2} in $\Omega$  we have 
	\begin{equation}
	\underset{\omega'}{\sup v} \leq C \,\underset{\omega'}{\inf v},
	\end{equation}
	where $C$ is a positive constant depending only on $d, p, \mathrm{dist}(\omega', \partial\omega)$, $\|A\|_{L^{\infty}(\gw;\R^{d\times d})}$, the ellipticity constant of $A$ in $\gw$, and $\|V\|_{M^q(p;\omega)}$ but not on $v$.
\end{theo}
The next result concerns the Harnack convergence principle for a  sequence of normalized positive solutions of equations of the form \eqref{p_laplace equ_2}  (cf. \cite[Proposition~2.11]{pinchover_psaradakis} where $A \in L_{\rm loc}^{\infty}(\Gw;\R^{d\times d})$ is  {\em fixed}, and $\{V_i\}_{i=1}^\infty  \!\subset \! M^q_{\loc}(p;\Omega_i)$ converges {\em strongly}  in $M^q_{\loc}(p;\Omega)$ to $\mathbb{V}\in M^q_{\loc}(p;\Omega)$).
\begin{prop}[Harnack convergence principle]\label{covergence_principle}
	Let $\{\Omega_i\}$ be a compact exhaustion of $\Omega$. Assume that $\{A_i\}_{i=1}^\infty$ is a sequence of symmetric and locally uniformly positive definite matrices such that  the local ellipticity constant does not depend on $i$, and   $\{A_i\}_{i=1}^\infty \!\subset  \! L^\infty_{\loc}(\Omega_i;\mathbb{R}^{d\times d})$ converges weakly in $L^\infty_{\loc}(\Omega;\mathbb{R}^{d\times d})$ to a matrix $\mathbb{A} \!\in  \! L^\infty_{\loc}(\Omega;\mathbb{R}^{d\times d})$. Assume also that $\{V_i\}_{i=1}^\infty  \!\subset \! M^q_{\loc}(p;\Omega_i)$ converges weakly in $M^q_{\loc}(p;\Omega)$ to $\mathbb{V}\in M^q_{\loc}(p;\Omega)$. 
	
	For each $i\geq 1$, let $v_i$ be a positive weak solution of the equation $Q_{p,A_i,V_i}(u)=0$ in $\Omega_i$ such that $v_i(x_0)=1$, where $x_0$ is a fixed reference point in $\Gw_1$.
	
	Then there exists $0<\beta<1$ such that, up to a subsequence, $\{v_i\}$ converges weakly in $W^{1,p}_{\loc}(\Gw)$ and in $C_{\loc}^\beta(\Omega)$ to a positive weak solution $v$ of the equation $Q_{p,\mathbb{A},\mathbb{V}}(u)=0$ in $\Omega$.
\end{prop}
\begin{proof}
	Since the sequence $\{A_i\}$ is locally uniformly elliptic and converges weakly in $L^\infty_{\loc}(\Omega;\mathbb{R}^{d\times d})$, it follows that that 
	$\|A_i\|_{L^\infty(\Gw'; \mathbb{R}^{d\times d})}\leq C$ for every $\Gw'\Subset\Omega$, and hence, $A_i$ are uniformly bounded in every $\Gw'\Subset\Omega$ expect for a set of measure zero.  By the definition of $v_i$ being a positive weak solution to $Q_{A_i,p,V_i}(v)=0$ in $\Omega_i$, we have 
	\begin{equation}\label{heqn_1}
	\int_{\Omega_i}|\nabla v_i|_{A_i}^{p-2}A_i\nabla v_i\cdot\nabla u\dx+\int_{\Omega_i}V_iv_i^{p-1}u\dx=0\,\,\,\text{for all}\,\,u\in W_0^{1,p}(\Omega_i).
	\end{equation}
	Also, by elliptic regularity,  $v_i$ are H\"older continuous for all $i\geq 1$. Fix $k\in \mathbb{N}$. Thus, for $u\in C_0^\infty(\Omega_k)$, by plugging $v_i|u|^p \in W_0^{1,p}(\Omega_k)$, $i\geq k$, as a test function in \eqref{heqn_1} we get
	\begin{equation*}
	\|\, |\nabla v_i|_{A_i}u\|^p_{L^p(\Omega_k)}\leq p \int_{\Omega_k}|\nabla v_i|_{A_i}^{p-1}|u|^{p-1}v_i|\nabla u|_{A_i}\dx+\int_{\Omega_k}|V_i|v_i^p|u|^p\dx.
	\end{equation*}
	For the first term of the right-hand side of the above equation, we apply Young's inequality: $p a b\leq \epsilon a^{p'}+[(p-1)/\epsilon]^{p-1}b^p$, $\epsilon\in(0,1)$, with $a=|\nabla v_i|_{A_i}^{p-1}|u|^{p-1}$ and $b=v_i|\nabla u|_{A_i}$. On the second term, we use the Morrey-Adams theorem (Theorem~\ref{theorem:uncertainty}). Then we arrive at 
	\begin{align*}
	(1-\epsilon)\||\nabla v_i|_{A_i}u\|^p_{L^p(\Omega_k)}\leq \left(\frac{p-1}{\epsilon}\right)^{p-1} &\|v_i|\nabla u|_{A_i}\|^p_{L^p(\Omega_k)} +\delta\|\nabla(v_iu)\|^p_{L^p(\Omega_k;\mathbb{R}^d)}  +C\|v_i u\|^p_{L^p(\Omega_k)}.
	\end{align*}
	Since the sequence $\{A_i\}$ is locally uniformly elliptic and bounded a.e., and by the inequality
	$$\|\nabla(v_iu)\|^p_{L^p(\Omega_k;\mathbb{R}^d)}\leq 2^{p-1}\left(\|v_i\nabla u\|^p_{L^p(\Omega_k;\mathbb{R}^d)}+\|u\nabla v_i\|^p_{L^p(\Omega_k;\mathbb{R}^d)}\right),$$
	we obtain the following estimates valid for $i\geq k$ and for any $u\in C_c^\infty(\Omega_k)$
	\begin{align}
	&\left((1\!-\!\epsilon)C_{\Omega_k}^p\!-\!2^{p-1}\delta C_{\Omega_k}^{-p}\right)\!\||\nabla v_i|u\|^p_{L^p(\Omega_k)}\label{heqn_2}  \\
	&\leq \left(\!\!\left(\!\frac{p\!-\!1}{\epsilon}\!\right)^{p\!-\!1}C^{-p}_{\Omega_k}\!+
	\!2^{p-1}\delta\!\right)\!\|v_i|\nabla u|\|^p_{L^p(\Omega_k)}\!\!+\!C(d,p,q,\delta,\|\mathbb{V}\|_{M^q(p;\Omega_{k\!+\!1})})\|v_i u\|^p_{L^p(\Omega_k)}.\nonumber
	\end{align}
	We now take an arbitrary $\omega\Subset\Omega$ and without loss of generality we assume that $x_0\in\omega$. Picking a subdomain  $\omega'\Subset\Omega$ such that $\omega\Subset\omega'$, we can find $k\geq 1$ such that $\omega'\Subset\Omega_k$. Then we choose $\delta <(1-\epsilon)2^{1-p}C^{2p}_{\Omega_k}$ and specialize $u\in C_c^\infty(\Omega_k)$ such that
	\begin{equation}\label{heqn_3}
	\text{supp}\{u\}\subset \omega',\,\,0\leq u\leq 1,\,\,u=1\,\text{in}\,\omega\,\,\text{and}\,\,|\nabla u|\leq \frac{1}{\text{dist}(\omega,\partial \omega')}\,\,\text{in}\,\,\omega'.
	\end{equation}
	Due to the local Harnack inequality \eqref{harnack_ineq}, the sequence $\{v_i\}_{i=1}^\infty$ of solutions is bounded  in $L^\infty(\omega)$. In fact, by elliptic regularity $\{v_i\}_{i=1}^\infty$ is bounded in $C^\ga(\omega)$, where $0<\ga\leq 1$. Moreover, by plugging $u$ as in \eqref{heqn_3} to the inequality \eqref{heqn_2}, we get
	$$\|\nabla v_i\|^p_{L^p(\omega;\mathbb{R}^d)}+\|v_i\|^p_{L^p(\omega)}\leq C(d,p,q,\epsilon,\delta, \text{dist}(\omega,\partial \omega'), C_{\Omega_k}, \|\mathbb{V}\|_{M^q(p;\Omega_{k+1})}),$$
	for all $i\geq k$. This implies that the sequence $\{v_i\}_{i=1}^\infty$ is bounded in $W^{1,p}(\omega)$. Hence up to a subsequence, still denoted by $\{v_i\}_{i=1}^\infty$, we obtain that  $\{v_i\}_{i=1}^\infty$ converges uniformly in $\gw$, and weakly to a nonnegative function $v\in W^{1,p}(\omega)\cap C^\ga(\gw)$ with $v(x_0)=1$. So, we have 
	\begin{align*}
	v_i \rightarrow v \quad  \text{ uniformly in } \omega,  \quad \text{ and } \nabla v_i  \rightharpoonup \nabla v\,\,\text{in}\,L^{p}(\omega;\mathbb{R}^d).
	\end{align*}
	We now show that $v$ is a weak solution of $Q_{p,\mathbb{A},\mathbb{V}}(u)=0$ in $\tilde{\omega}\Subset \omega$ such that $x_0\in \tilde{\omega}$. 
	Using the uniform convergence in $\gw$ of $v_i$ to $v$, we obtain
		\begin{align*}
	\left|\int_{\omega}(V_iv_i^{p-1}\vgf -\mathbb{V}v^{p-1}\vgf)\dx\right|
	\leq C\|v_i-v\|_{L^\infty(\gw)} \int_\omega |V_i|\dx+\left|\int_{\omega} (V_i -\mathbb{V})v^{p-1}\vgf \dx\right| .
	\end{align*}
	
	Since the sequence $\{V_i\}$ converges weakly to $\mathbb{V}$ in $M^q_{\loc}(p;\Omega)$, it is bounded in $L_{\loc}^1(\Omega)$, therefore, the first term tends to zero, while the second term tends to zero by the weak convergence of $\{V_i\}$ to $\mathbb{V}$.  Hence, 
	\begin{equation}\label{heqn_4}
	\int_{\omega}V_iv_i^{p-1}\vgf\d x \rightarrow \int_{\omega}\mathbb{V}v^{p-1}\vgf\d x \qquad \text{for all } \vgf\in C_c^\infty(\omega).
	\end{equation}
		It remains to show that 
	\begin{equation}\label{heqn_w}
	\xi_i:= |\nabla v_i|^{p-2}_{A_i}A_i\nabla v_i \rightharpoonup_{i\rightarrow\infty} |\nabla v|^{p-2}_{\mathbb{A}}\mathbb{A}\nabla v :=\xi\,\,\,\text{in}\,L^{p'}(\tilde{\omega};\mathbb{R}^d).
	\end{equation}
	To prove this claim, it is enough to prove that $\xi_i\to \xi$ a.e., and that  $\{\xi_i\}$ is bounded in $L^{p'}(\tilde{\omega};\mathbb{R}^d)$  (see, \cite{Mz}   and \cite[Lemma~3.73]{Heinonen}).
		The boundedness of $\{\xi_i\}$ in $L^{p'}(\tilde{\omega};\mathbb{R}^d)$ clearly follows  from the boundedness of $\{\nabla v_i\}$ in   $L^{p}(\tilde{\omega};\mathbb{R}^d)$. So, we need to prove the a.e. convergence of $\{\xi_i\}$ to $\xi$.

		
	By our assumption,	$\{A_i\}_{i=1}^\infty \subset L^\infty_{\loc}(\Omega_i;\mathbb{R}^{d\times d})$ converges weakly in $ L^\infty_{\loc}(\Omega;\mathbb{R}^{d\times d})$ to a matrix $\mathbb{A}\in L^\infty_{\loc}(\Omega;\mathbb{R}^{d\times d})$. Then let us consider $u$ as in \eqref{heqn_3} but with $\omega$ and $\omega'$ replaced by $\tilde{\omega}$ and $\omega$, respectively. So, by plugging $u(v_i-v)$ as a test function in \eqref{heqn_1}, we obtain
	\begin{equation}\label{heqn_5}
	\int_{\omega}u\xi_i\cdot \nabla(v_i-v)\d x= -\int_{\omega}(v_i-v)\xi_i\cdot \nabla u\dx -\int_\omega V_iv_i^{p-1}u(v_i-v)\dx.
	\end{equation}
	For the first integral on the right-hand side of \eqref{heqn_5}, apply H\"{o}lder's inequality to get
	\begin{align*}
	\left|-\int_{\omega}(v_i-v)\xi_i\cdot \nabla u\d x\right|&
	 \leq C_\omega^{p/{p'}}\|(v_i-v)\nabla u\|_{L^p(\omega;\mathbb{R}^d)}\|\nabla v_i\|^{p/{p'}}_{L^p(\omega;\mathbb{R}^d)}\\
	& \leq C(p,C_\omega, \text{dist}(\tilde{\omega},\partial \omega))\|(v_i-v)\|_{L^p(\omega)}\|\nabla v_i\|^{p/{p'}}_{L^p(\omega;\mathbb{R}^d)} \to_{i\to\infty} 0,
	\end{align*}
	 since $\|\nabla v_i\|_{L^p(\omega;\mathbb{R}^d)}$ are uniformly bounded and $v_i\rightarrow v$ in $L^p(\omega)$. A similar argument leading to \eqref{heqn_4} implies that the second integral on the right-hand side of \eqref{heqn_5} also converges to $0$. Thus,
	\begin{equation}\label{heqn_6}
	\int_{\omega}u\xi_i\cdot \nabla(v_i-v)\d x \rightarrow_{i\rightarrow\infty} 0.
	\end{equation}
	Notice that 
		\vspace{-3mm} 
	\begin{align*}
	&(\xi_i-\xi)\cdot(\nabla v_i-\nabla v)=(|\nabla v_i|^{p-2}_{A_i}A_i\nabla v_i-|\nabla v|^{p-2}_{A_i}A_i\nabla v)\cdot(\nabla v_i-\nabla v)\\
	&+\!(|\nabla v|^{p\!-\!2}_{A_i}\!\!A_i\!\nabla v\!-\!|\nabla v|^{p\!-\!2}_\mathbb{A} \!\mathbb{A}\nabla v)\!\!\,\cdot\,\!\!(\nabla v_i \!-\! \nabla v)
	\!\geq\! (|\nabla v|^{p\!-\!2}_{A_i}\!A_i\!\nabla v\!- \!|\nabla v|^{p\!-\!2}_\mathbb{A}\! \mathbb{A}\nabla v)\!\cdot\!(\nabla v_i \!\!-\! \nabla v).
	\end{align*}
	Since $A_i$ converges weakly to $\mathbb{A}$ in $L^\infty_{\loc}(\Omega, \mathbb{R}^{d\times d})$, it follows \cite[Proposition 2.9]{moreira} that $A_i\rightarrow \mathbb{A}$ a.e.. Therefore, 
	$$|\nabla v(x)|^{p-2}_{A_i(x)}A_i(x)\nabla v(x) \rightarrow |\nabla v(x)|^{p-2}_{\mathbb{A}(x)} \mathbb{A}(x)\nabla v(x)\quad \mbox{for a.e. in $\omega$, and  also }$$  
	\begin{align*}
	\left||\nabla v|^{p-2}_{A_i}A_i\nabla v-|\nabla v|^{p-2}_\mathbb{A} \mathbb{A}\nabla v\right|^{p'} \leq 2^{p'-1}(|\nabla v|^{p}_{A_i}  +|\nabla v|^p_{\mathbb{A}})\leq C|\nabla v|^{p},
	\end{align*}
	since the sequence $\{A_i\}$ is bounded a.e in $\omega$. Thus, the dominated convergence theorem implies 
	$$\lim_{i\rightarrow \infty}\int_{{\omega}}\left||\nabla v|^{p-2}_{A_i}A_i\nabla v-|\nabla v|^{p-2}_\mathbb{A} \mathbb{A}\nabla v\right|^{p'}\d x=0.$$
	Therefore, the H\"{o}lder inequality and the boundedness of $\nabla v_i$ in $L^p(\omega;\mathbb{R}^d)$ implies 
	\begin{equation}\label{heqn_7}
	\left |\int_{\omega}u(|\nabla v|^{p-2}_{A_i}A_i\nabla v-|\nabla v|^{p-2}_\mathbb{A}\mathbb{A}\nabla v)\cdot(\nabla v_i-\nabla v)\d x\right|\rightarrow_{i\rightarrow\infty}0.
	\end{equation}
	Now by using above vectors inequality we get
	\begin{align*}
	0 \! &\leq \int_{\tilde{\omega}}\!\left[(\xi_i-\xi)\cdot(\nabla v_i-\nabla v)- (|\nabla v|^{p-2}_{A_i}A_i\nabla v-|\nabla v|^{p-2}_\mathbb{A}\mathbb{A}\nabla v)\cdot(\nabla v_i-\nabla v)\right]\!\! \d x  \\
	&\leq \int_{{\omega}}\! u\left[\!(\xi_i-\xi)\cdot(\nabla v_i-\nabla v)- (|\nabla v|^{p-2}_{A_i}A_i\nabla v-|\nabla v|^{p-2}_\mathbb{A} \mathbb{A}\nabla v)\cdot(\nabla v_i-\nabla v)\!\right] \!\!\d x\\
	&=  \int_{\omega}\!u(\xi_i-\xi)\!\cdot\!(\nabla v_i-\nabla v)\d x \! - \!  \int_{\omega}u(|\nabla v|^{p-2}_{A_i}A_i\nabla v-|\nabla v|^{p-2}_\mathbb{A} \mathbb{A}\nabla v)\!\cdot\!(\nabla v_i-\nabla v)\d x \rightarrow_{i\rightarrow\infty}0,
	\end{align*}
	where we have used \eqref{heqn_6}, \eqref{heqn_7} and $\nabla v_i  \rightharpoonup \nabla v\,\,\text{in}\,L^{p}(\omega;\mathbb{R}^d)$. It follows that
	\begin{align}\label{heqn_8}
	&\lim_{i\rightarrow \infty} \!\int_{\tilde{\omega}} (|\nabla v_i|^{p-2}_{A_i}A_i\nabla v_i -|\nabla v|^{p-2}_{A_i}A_i\nabla v)\cdot(\nabla v_i-\nabla v)\dx\\
	&\!=\!\lim_{i\rightarrow \infty}	\!\!\int_{\tilde{\omega}}\!\!\left[\!(\xi_i-\xi)\cdot(\nabla v_i\!-\!\nabla v)- (|\nabla v|^{p-2}_{A_i}A_i\nabla v \!-\! |\nabla v|^{p-2}_\mathbb{A}\mathbb{A}\nabla v)\cdot(\nabla v_i-\nabla v)\!\right]\!\!\d x\!=\!0.\nonumber
	\end{align}
		
	To prove the claim \eqref{heqn_w}, we proceed as in \cite[Lemma~3.73]{Heinonen}. Denote
	\begin{align*}
	D_i = (|\nabla v_i|^{p-2}_{A_i}A_i\nabla v_i -|\nabla v|^{p-2}_{A_i}A_i\nabla v)\cdot(\nabla v_i-\nabla v).
	\end{align*} 
	Since $D_i$ is a nonnegative function, \eqref{heqn_8} implies that $D_i\rightarrow 0 \text{ in } L^1(\tilde{\omega})$. 
	Extracting a subsequence we have $D_i\rightarrow 0$ a.e. in $\tilde{\omega}$. Therefore, there exists a subset $Z$ of $\tilde{\omega}$ of zero measure such that for $x\in \tilde{\omega}\setminus Z$  we have $D_i(x)\rightarrow 0$. 
	
	Fix $x\in \tilde{\omega}\setminus Z$. Without loss of generality, we may assume that $|\nabla v(x)|<\infty$. Since $A_i$ are locally uniformly elliptic and bounded, we have 
	\begin{align*}
	D_i(x) & \geq  |\nabla v_i(x)|^p_{A_i} +|\nabla v(x)|^p_{A_i}- (|\nabla v(x)|_{A_i}|\nabla v_i(x)|^{p-1}_{A_i}+ |\nabla v_i(x)|_{A_i}|\nabla v(x)|^{p-1}_{A_i})\\
	& \geq C^p_{\tilde{\omega}} |\nabla v_i(x)|^p- C^{-p}_{\tilde{\omega}}(|\nabla v(x)||\nabla v_i(x)|^{p-1}+ |\nabla v_i(x)||\nabla v(x)|^{p-1})\\
	& \geq C^p_{\tilde{\omega}} |\nabla v_i(x)|^p -C(|\nabla v_i(x)|^{p-1}+ |\nabla v_i(x)|),
	\end{align*}
	where $C=\max (|\nabla v(x)|, |\nabla v(x)|^{p-1})C^{-p}_{\tilde{\omega}}$.   From the above inequality, it readily follows that $|\nabla v_i(x)|$ is uniformly bounded with respect to $i$, since $D_i(x)\rightarrow 0$.
	
	Let $\gh$ be a limit point of $\nabla v_i(x)$. Then $|\gh|<\infty$ and 
	\begin{align*}
	0& =\lim_{i\rightarrow \infty} (|\nabla v_i(x)|^{p-2}_{A_i}A_i\nabla v_i(x) -|\nabla v(x)|^{p-2}_{A_i}A_i\nabla v(x))\cdot(\nabla v_i(x)-\nabla v(x))\\
	&= (|\gh|^{p-2}_\mathbb{A}\mathbb{A}\gh-|\nabla v(x)|^{p-2}_\mathbb{A}\mathbb{A}\nabla v(x))\cdot(\gh -\nabla v(x)).
	\end{align*}
	This implies that $\gh=\nabla v(x)$. Thus we get $\nabla v_i(x)\rightarrow\nabla v(x)$ for every $x\in \tilde{\omega}\setminus Z$, i.e., $\nabla v_i \rightarrow \nabla v$ a.e. in $\tilde{\omega}$ and $|\nabla v_i(x)|^{p-2}_{A_i}A_i\nabla v_i(x)\rightarrow |\nabla v(x)|^{p-2}_\mathbb{A}\mathbb{A}\nabla v(x)$ a.e. in $\tilde{\omega}$. 
	Recall that the $L^{p'}$-norm of $\{|\nabla v_i|^{p-2}_{A_i}A_i\nabla v_i\}$ is bounded in $\tilde{\omega}$, therefore, \eqref{heqn_w} follows.
\end{proof}
Finally, we formulate a weak comparison principle  (WCP) for the case $A\in L_{\loc}^{\infty}(\Gw;\R^{d\times d})$, and  $V\in M^q_{\loc}(p;\Omega)$. For the proof see Theorem~5.3 in \cite{pinchover_psaradakis}.
\begin{theo}[Weak comparison principle]\label{weak_comparison}
	Let $\omega\Subset\Omega$ be a bounded Lipschitz domain, and let  $A$  and $V$ satisfy Assumptions~\ref{assump}. Assume that the equation $Q_{p,A,V}(u)=0$ admits a positive solution in $W_{\loc}^{1,p}(\Omega)$ and suppose that $u_1, u_2\in W^{1,p}(\omega)\cap C(\bar{\omega})$ satisfy the following inequalities
	\begin{equation*}
	\begin{cases*}
	Q_{p,A,V}(u_2)\geq 0\,\,\text{in}\,\,\omega,\\
	u_2 >0\,\,\,\,\,\text{on}\,\,\partial\omega,
	\end{cases*}\quad \mbox{and}\quad 
	\begin{cases*}
	Q_{p,A,V}(u_1)\leq Q_{p,A,V}(u_2) \,\,\text{in}\,\,\omega,\\
	u_1\leq u_2\,\,\,\,\,\text{on}\,\,\partial\omega.
	\end{cases*}
	\end{equation*}
	Then $u_1 \leq u_2$ in $\omega$.
\end{theo}
\section{Fuchsian-type singularity}\label{sec_fuchs}
We introduce the notion of Fuchsian-type singularity with respect to the equation $Q_{p,A,V}(u) = 0$.  We allow the domain $\Omega$ to be unbounded and with nonsmooth boundary, and the singular point to be  $\zeta=\infty$. Thus, it is convenient to consider the one-point compactification $\hat{\mathbb{R}}^d$ of $\R^d$, i.e.,  $\hat{\mathbb{R}}^d:=\mathbb{R}^d\cup\{\infty\}$. By $\hat{\Omega}$, we denote the closure of $\Omega$ in $\hat{\mathbb{R}}^d$. This should not be confused with the one-point compactification of a domain $\Gw$ which is also considered in the sequel. In the latter topology, {\em a neighbourhood of infinity in $\Omega$} is a set of the form $\Omega\setminus K$, where $K\Subset\Omega$.  

 
{\bf Throughout this paper, we assume that singular point $\zeta \in  \partial\hat\Omega$ is either $0$ or $\infty$, and that $\gz$ is an isolated component of  $\partial\hat\Omega$.}  


With some abuse of notation, we write 
	\vspace{-3mm} 
$$a\to\gz \;\;\mbox{if} \;\;  
\begin{cases} 
a \to 0 & \mbox{in } \R \, \mbox{ and } \gz =0,\\
a \to \infty & \mbox{in  }  \R \, \mbox{ and }  \gz = \infty.
\end{cases} $$
  

We extend the definition of pointwise Fuchsian-type singularity (see \eqref{eq_pFs}).
\begin{defin}[Fuchsian singularity]\label{fuchsian}{\em 
	Let $\Omega$ be a domain in $\mathbb{R}^d$, and $A$  and $V$ satisfy Assumptions~\ref{assump}.   Let $\zeta\in \partial\hat{\Omega}$ be an isolated point of $\partial\hat{\Omega}$, where $\zeta = 0$ or $\zeta =\infty$. We say that  the operator $Q_{p,A,V}$ has a {\em Fuchsian-type singularity at} $\zeta$ (in short,  Fuchsian singularity at $\zeta$) if the following two conditions are satisfied: 
	\begin{leftbrace} 
			\begin{itemize}
				\item[(1)] The matrix $A$ is bounded and uniformly elliptic in a punctured neighbourhood  $\Omega'\subset \Omega$ of  $\gz$,
				that is, there is a positive constant $C$ such that, 
				\be\label{ell}
				C^{-1}|\xi|^2\leq |\xi|_{A}^2 \leq C^{-1}|\xi|^2
				\qquad \forall x\in\Gw'\mbox{ and }\xi=(\xi_1,...,\xi_d)\in\R^d.
				\ee
				\item[(2)] There exists a positive constant $C$ and $R_0>0$ such that 
				\begin{equation}\label{fuchsian_eqn}
				\begin{cases}
				\||x|^{p-{d}/{q}}\, V\|_{M^q(p;\mathcal{A}_R)}\leq C  & \text{if } p\neq d, \\[2mm]
				\|V\|_{M^q(d;\mathcal{A}_R)}\leq C & \text{if } p=d, 
				\end{cases}
				\end{equation}
				for all $0\!<\!R<\!R_0$ if $\gz\!=\!0$, and $R\!>\!R_0$ if $\gz\!=\!\infty$, where $\mathcal{A}_R\!:=\!\{x\mid  {R}/{2}\!\leq \! |x|\! < \!{3R}/{2} \}.$  			
			\end{itemize}
	\end{leftbrace}
		 }
\end{defin}
\begin{defin}{\em 
		 A set $\mathcal{A}\subset \Gw$ is called  an {\em essential set with respect to the singular point $\zeta$} if  there exist real numbers $0<a<1<b<\infty$, and a sequence of positive numbers $\{R_n\}$ with $R_n\rightarrow \zeta$ such that $\mathcal{A}=\cup \mathcal{A}_n, \mbox{ where }    \mathcal{A}_n=\{x\in\Omega\, |\,
	a R_n<|x|<b R_n\}$.
}
\end{defin}
\begin{rem}\label{remrk1}{\em 
	Similar to the linear case \cite{pinchover}, it turns out that it is sufficient to assume that inequality \eqref{fuchsian_eqn} is satisfied only on some {\it essential} subset of a neighbourhood of $\zeta$. More precisely, instead of \eqref{fuchsian_eqn}, we may assume that for some essential set  $\mathcal{A}\!\subset \!\Gw$ with respect to  $\zeta$ 
	\begin{equation}\label{fuchsian_eqn1}
	\begin{cases}
	\||x|^{p-{d}/{q}}\, V\|_{M^q(p;\mathcal{A}_n)}\leq C  & \text{if } p\neq d, \\
	\|V\|_{M^q(d;\mathcal{A}_n)}\leq C & \text{if } p=d, 
	\end{cases}
		\qquad \mbox{where $C$ is independent of $n$}. 
	\end{equation}
 }
\end{rem}
\begin{example}{\em 
	Let $\Omega=\mathbb{R}^d\setminus \{0\}$ and fix $1<p<\infty$. Consider the $p$-Laplace operator with the Hardy potential $V(x)=\lambda |x|^{-p}$
		\vspace{-3mm} 
	\begin{equation}\label{hardy_eqn}
	-\Delta_{p} u-\lambda \frac{|u|^{p-2}u}{|x|^p}=0\,\,\,\text{in}\,\,\Omega,
	\end{equation}
	where $\gl\leq C_H:= |(p-d)/p|^{p}$ is the Hardy constant. A straightforward computation shows that  \eqref{hardy_eqn} has Fuchsian singularity both at $\zeta=0$ and $\gz=\infty$ (in the sense of definition~\eqref{fuchsian}).
}
	\end{example}
	The above example implies:
\begin{example}{\em 
	Let $\Omega=\mathbb{R}^d\setminus \{0\}$ and fix $1<p<\infty$. Consider the operator $ Q_{p,A,V}=-\Delta_{p,A}(u)+V|u|^{p-2}u$, and assume that the matrix $A$ satisfies \eqref{ell} in $\Gw$, and   $|V(x)|\leq  C|x|^{-p}$ in $\Gw$.  Then  $Q_{p,A,V}$ has  Fuchsian singularity at the origin and at infinity. 
}
\end{example}
We now present a dilation process which uses the quasi-invariance of Fuchsian equations of the form \eqref{p_laplace equ_2} under the scaling $x\mapsto Rx$ for $R>0$. Let $A_R$ and $V_R$ be the scaled matrix and potential defined by 
	\vspace{-3mm} 
\begin{equation*}
A_R(x):= A(Rx), \quad V_R(x):= R^pV(Rx)\qquad x\in\,\Omega/R.
\end{equation*}
Consider the annular set $\mathcal{A}_R=(B_{3R/2}\setminus \bar{B}_{R/2} )\cap \Omega$. By our assumption that $\gz$ is an isolated singular point it follows that  $\mathcal{A}_R/R$ is fixed annular set $\tilde{\mathcal{A}}$ for $R$ `near' $\gz$, and  for such $R$ we have
\begin{equation}\label{eq_AR}
\|V_R\|_{M^q(p;\tilde{\mathcal{A}})} =\|V_R\|_{M^q(p;\mathcal{A}_R/R)}= R^{p-d/q}\|V\|_{M^q(p;\mathcal{A}_R)} \leq C,
\end{equation}
for $p\neq d$ and while for $p=d$, 
\begin{equation}\label{eq_AR1}
\|V_R\|_{M^q(d;\tilde{\mathcal{A}})}=\|V_R\|_{M^q(d;\mathcal{A}_R/R)}=\|V\|_{M^q(d;\mathcal{A}_R)}\leq C.
\end{equation}

\medskip

Let $Y:=\displaystyle{\lim_{n\rightarrow\infty} \Omega/R_{n}}$. Note that since by our assumptions, $\gz$ is an isolated component of  $\partial\hat\Omega$, it follows that $Y=(\R^d)^*=\R^d\setminus \{0\}$. 

The {\em limiting dilation process} is defined as follows. Let $\zeta =0$ or $\zeta=\infty$, and assume that there is a sequence $\{R_n\}$ of positive numbers satisfying $R_n\rightarrow \zeta$  such that
\begin{align}\label{LDP}
	\begin{cases}
	A_{R_n}\xrightarrow{n\rightarrow \infty} \mathbb{A}&  \mbox{ in the weak topology of }
L_{\mathrm{loc}}^{\infty}(Y;\R^{d\times d}), \mbox{ and } \\
V_{R_n}\xrightarrow{n\rightarrow \infty} \mathbb{V}&   \mbox{ in the weak topology of }
  M^q_{\loc}(p;Y).
  \end{cases}
\end{align}
 Motivated by the Harnack convergence principle (see, Proposition~\ref{covergence_principle}),  we define the {\em limiting dilated equation} with respect to equation \eqref{p_laplace equ_2} and the sequence $\{R_n\}$ that satisfies \eqref{LDP} by 
\begin{equation}\label{dilated_eqn}
\mathcal{D}^{\{R_n\}}(Q)(w):= -\Delta_{p,\mathbb{A}}(w)+\mathbb{V}|w|^{p-2}w=0\,\,\text{on}\,\,Y.
\end{equation}
The following proposition establishes a key invariance property of the limiting dilation process.  
\begin{prop}\label{fuchsian_prop} 
	Let   $A, V,$ satisfy Assumptions~\ref{assump}. Suppose that the quasilinear equation \eqref{p_laplace equ_2} has a Fuchsian singularity at $\zeta\in\{0,\infty\} \subset \partial\hat{\Omega}$, and let $\mathcal{D}^{\{R_n\}}(Q)(w)$  be a limiting dilated operator corresponding to a sequence ${R}_n\rightarrow\zeta$. Then the equation $\mathcal{D}^{\{R_n\}}(Q)(w)=0$ in $Y$ has Fuchsian singularity at $\zeta$.
\end{prop}
\begin{proof}
	 	It is trivial to verify that the proposition holds true when $p=d$. Now for $p\neq d$, by Remark~\ref{remrk1}, there exists $C>0$ and an {\em essential} set $\mathcal{A}=\cup \mathcal{A}_n, \mbox{ where }    \mathcal{A}_n=\{x\in\Omega\, |\,aR_n<|x|<bR_n\}$ such that
	 $$\||x|^{p-d/q}V\|_{M^q(p;\mathcal{A}_n)}\leq C  \qquad \forall n\in \N.$$
	 We claim that 
	 $$\||x|^{p-d/q}\mathbb{V}\|_{M^q(p;\mathcal{A}/R_n)}\leq C \qquad \forall n\in \N.$$
	 Recall that for each $n$, $\mathcal{A}_n/R_n$ is a fixed annular set $\tilde{A}=\{x \,|\, a<|x|<b\}$. Assume that $p<d$, so, $p-d/q>0$. Then we have 
	 \begin{align*}
	 &\||x|^{p-d/q}\mathbb{V}\|_{M^q(p;\mathcal{A}_n/R_n)} \leq b^{p-d/q}\|\mathbb{V}\|_{M^q(p;\mathcal{\tilde{A}})}
	  \leq b^{p-d/q} \underset{n\rightarrow\infty}{\lim \inf}\,\|V_{R_n}\|_{M^q(p;\mathcal{A}_n/R_n)} \\[2mm]
	  &= b^{p-d/q} \underset{n\rightarrow\infty}{\lim \inf}\,\|R_n^{p-d/q}V\|_{M^q(p;\mathcal{A}_n)}
	 \leq \left(\frac{b}{a}\right)^{p-d/q} \underset{n\rightarrow\infty}{\lim \inf}\,\||x|^{p-d/q}V\|_{M^q(p;\mathcal{A}_n)}
	  \leq C.
	 \end{align*}
	 For $p>d$,  $M^q_{\text{loc}}(\Omega)=L^1_{\text{loc}}(\Omega)$, and similarly we get 
	$ \||x|^{p-d}\mathbb{V}\|_{L^1(\mathcal{A}_n/R_n)} \leq C$.
\end{proof}
\begin{defin}{\em 
		Let $\mathcal{G}_\zeta$ be the germs of all positive solutions $u$ of the equation $Q_{p,A,V}(w)=0$ in relative punctured neighbourhoods  of $\zeta$.
	We say that  $\zeta$ is a regular point of the above equation if for any two positive solutions $u,v\in \mathcal{G}_\zeta$
	$$\underset{x\in\Omega'}{\displaystyle{\lim_{x\rightarrow \zeta}}}\frac{u(x)}{v(x)} \quad \text{exists in the wide sense}.$$
}
\end{defin}
Next, we define two types of positive solutions of minimal growth which was introduced by Agmon \cite{agmon} for the linear case, and was later extended to $p$-Laplacian type equations in \cite{frass_pinchover,PR, PT}.
\begin{defin}{\em 
		(1) Let $K_0$ be a compact subset of $\Omega$. A positive solution $u$ of the equation $Q_{p,A,V}(u)=0$ in $\Omega\setminus K_0$ is said to be a {\it positive solution of minimal growth in a neighbourhood of infinity in} $\Omega$ if for any smooth compact subset $K$ of $\Omega$ with $K_0\Subset \mbox{int}\, K$ and any positive supersolution $v\in C(\Omega\setminus\mbox{int}\,K)$ of the equation $Q_{p,A,V}(u)=0$ in $\Omega\setminus K$, we have
		$$u\leq v\quad\mbox{ on } \partial K \quad \Rightarrow \quad u\leq v \quad \mbox{ in }\Omega\setminus K.$$
		
		(2) A positive solution of the equation $Q_{p,A,V}(u)=0$ in $\Omega$ which has minimal growth in a neighbourhood of infinity in $\Omega$ is called a {\it ground state}  of $Q_{p,A,V}$ in $\Gw$.
		 
		(3) Let $\zeta\in \partial\hat{\Omega}$ and let $u$ be a positive solution of the equation $Q_{p,A,V}(u)=0$ in $\Omega$. Then $u$ is said to be a {\it positive solution of minimal growth in a neighbourhood of} $\partial\hat{\Omega}\setminus\{\zeta\}$ if for any relative neighbourhood $K\Subset \hat{\Omega}$ of $\zeta$ such that $\Gamma:= \partial K\cap \Omega$ is smooth, and any positive supersolution $v\in C((\Omega\setminus K)\cup \Gamma)$ of the equation $Q_{p,A,V}(u)=0$ in $\Omega\setminus K$, we have 
		$$u\leq v\quad \mbox{ on } \Gamma \quad \Rightarrow \quad u\leq v\quad \mbox{ in } \Omega\setminus K.$$
	}
\end{defin}
\begin{prop}\label{minimum_growth1}
	Suppose that $\mathcal{Q}_{p,A,V}\geq 0$ in $\Gw$, and  $Q_{p,A,V}$ has an isolated Fuchsian  singularity at $\zeta\in\partial\hat{\Omega}$, where $\gz\in\{0,\infty\}$. Then equation \eqref{p_laplace equ_2} admits a positive solution in $\Omega$ of minimal growth in a neighbourhood of $\partial\hat{\Omega}\setminus\{\zeta\}$.
\end{prop}
\begin{proof}
	Let $\zeta=0$. By \cite[Theorem 5.7]{pinchover_psaradakis}, for any $x_0\in \Omega$, the equation $Q_{p,A,V}(u)=0$ admits a positive solution $u_{x_0}$ of the equation $Q_{p,A,V}(u)=0$ in $\Omega\setminus\{x_0\}$ of minimal growth in a neighbourhood of infinity in $\Omega$. Note that the proof of \cite[Theorem 5.7]{pinchover_psaradakis} applies also in case that $x_0\in \partial \Gw$ is an isolated singular point of $\partial \Gw$. Hence,  \eqref{p_laplace equ_2} admits a positive solution in $\Omega$ of minimal growth in a neighbourhood of $\partial\hat{\Omega}\setminus\{0\}$.
	
	
	Now consider the case when $\zeta=\infty$. Let $\{x_n \}\subset \Gw$ be a sequence such that $x_n\to\infty$. Fix a reference point $x_0\in \Gw$, and  a compact smooth  exhaustion $\{\Gw_k\}$ of $\Gw$.  For $n\in\N$,  denote by $u_{x_n}$  a positive solution of the equation $Q_{p,A,V}(u)=0$ in $\Omega\setminus\{x_n\}$ of minimal growth in a neighbourhood of infinity in $\Omega$, and let  $v_{x_n}(x):= u_{x_n}(x)/u_{x_n}(x_0)$. By the Harnack convergence principle, up to a subsequence, $v_{x_n}$ converges locally uniformly to $v$ which is a positive solution of the equation $Q_{p,A,V}(u)=0$ in $\Omega$.
	
	We claim that $v$ is  a positive solution in $\Omega$ of minimal growth in a neighbourhood of $\partial\hat{\Omega}\setminus\{\infty\}$. Indeed, let $K\subset \hat{\Omega}$ be a punctured neighborhood of $\infty$ with smooth boundary, and let $w$ be a positive supersolution of the equation $Q_{p,A,V}(u)=0$ in $\Gw'=\Gw \setminus K$, such that $v\leq w$ on $\partial K$. Then for any $\vge>0$ there exists $n_\vge$ such that $v_{x_n}\leq (1 +\vge)w$ on $\partial K$ for all $n\geq n_{\vge}$.
	
	Recall that by the construction of $v_{x_n}$,  for a fixed $n$ we have, $v_{x_n}=\lim_{k\to \infty} v_{n,k}$, where $v_{n,k}$ restricted to $\Gw_k\cap \Gw'$ is a positive solution of the equation $Q_{p,A,V}(u)=0$ which vanishes on $\partial \Gw_k\cap \Gw'$. Therefore, by the weak convergence principle, 
	$v_{x_n}\leq (1 +\vge)w $ in $\Gw_k\cap \Gw'$, for all $n\geq n_{\vge}$. and therefore,  $v\leq w +\vge$ in $\Gw'$. Since $\vge$ is arbitrarily small, we have $v\leq w$ in $\Gw'$. Thus, $v$ is a positive solution in $\Omega$ of minimal growth in a neighbourhood of $\partial\hat{\Omega}\setminus\{\zeta\}$. 
\end{proof}
We extend Conjecture~1.1 in \cite{frass_pinchover} to the more general setting considered in the present paper. Our main goal is to prove it under some further relatively mild assumptions.
\begin{conjecture}\label{main_conj}
	Assume that Equation \eqref{p_laplace equ_2} has a Fuchsian-type isolated singularity at $\zeta\in \pd \hat \Omega$ and admits a (global) positive solution. Then
	\begin{enumerate}
		\item[(i)] $\zeta$ is a regular point of equation \eqref{p_laplace equ_2}.
		\item[(ii)] Equation~\eqref{p_laplace equ_2} admits a unique (global) positive solution of minimal growth in a neighborhood of $\partial \hat{\Gw} \setminus \{\zeta\}$.
	\end{enumerate}
\end{conjecture}
Next, we recall the notions of subcriticality and criticality (for more details see \cite{pinchover_psaradakis}). 
\begin{defin}\label{def:critical}\em{  Assume that $\mathcal{Q}_{p,A,V} \geq 0$ in $\Gw$. Then $\mathcal{Q}_{p,A,V}$ is called {\em subcritical} in $\Gw$ if there exists a nonzero nonnegative function $W\in  M^q_{\rm loc}(p;\Gw)$ such that
	\be\label{ineq:subcriticality}
	\mathcal{Q}_{p,A,V}(\gf)
	\geq
	\int_\Gw  W|\gf|^p\dx\qquad \mbox{for all } \gf\in C_c^\infty(\Gw).
	\ee
	If this is not the case, then $\mathcal{Q}_{p,A,V}$ is called {\em critical} in $\Gw$.}
\end{defin}
\begin{theorem}[\cite{pinchover_psaradakis}]\label{theorem:main} Let $\mathcal{Q}_{p,A,V}\geq 0$ in $\Gw$. Then the following assertions are equivalent: 
	\begin{itemize}
		\item[(i)] $\mathcal{Q}_{p,A,V}$ is critical in $\Gw$.
		\item[(ii)] The equation $Q_{p,A,V}(u)=0$  in $\Gw$ admits a unique positive supersolution.
		\item[(iii)] The equation $Q_{p,A,V}(u)=0$  in $\Gw$ admits a ground state $\phi$.
	\end{itemize}
\end{theorem}
The following uniform Harnack inequality near an isolated Fuchsian singular point $\gz$ is a key ingredient for obtaining regularity results at  $\gz$. 
\begin{theo}[Uniform Harnack inequality]\label{uniform_harnack}
	Let  $A$  and $V$ satisfy Assumptions~\ref{assump}, and assume that $Q=Q_{p,A,V}$  has an isolated Fuchsian singularity at $\zeta  \in \partial\hat{\Omega}$, where $\zeta =0$ or $\zeta=\infty$. Let $u, v\in \mathcal{G}_\zeta$. Consider the annular set $\mathcal{A}_r:= (B_{3r/2}\setminus\bar{B}_{r/2})\cap \Omega'$, where $\Gw'$ is a punctured neighbourhood of $\gz$. Denote
	\vspace{-3mm} 
	$$\mathbf{a}_r:= \underset{x\in \mathcal{A}_r}{\inf} \frac{u(x)}{v(x)}\,,\hspace*{1cm}\mathbf{A}_r:=\underset{x\in \mathcal{A}_r}{\sup} \frac{u(x)}{v(x)}\,.$$
	Then there exists $C>0$ independent of $r$,  $u$ and $v$ such that 
	$$ \mathbf{A}_r\leq C\mathbf{a}_r\qquad \forall \,  r \mbox{ near } \zeta.$$
\end{theo}
\begin{proof}
	Fix positive solutions $u$ and $v$ in $\Gw'\subset \Gw$, a fixed punctured neighbourhood of $\gz$. For $r>0$, let us consider the annular set $\tilde{\mathcal{A}}_r:= (B_{2r}\setminus\bar{B}_{r/4})\cap \Omega'$. Since $\zeta=0$ (respectively, $\zeta=\infty$) is an isolated singular point, then for $r<r_0$ (respectively, $r>r_0$) $\mathcal{A}_r/r$ and $\tilde{\mathcal{A}}_r/r$ are fixed annulus $\mathcal{A}$ and $\tilde{\mathcal{A}}$ respectively and $\mathcal{A}\Subset\tilde{\mathcal{A}}$.\\
	Now for such $r$, we define $u_r(x):=u(rx)$ for $x\in\Omega'/r$. Then the function $u_r$ is a positive solution of the equation 
	\begin{equation}
	Q_r[u_r]:=-\Delta_{p,A_r}(u_r)+ V_r(x)|u_r|^{p-2}u_r=0 \,\,\text{in}\,\,\tilde{\mathcal{A}},
	\end{equation}
	where $A_r(x):= A(rx)$ and $V_r=r^{p}V(rx)$. Similarly, $v_r(x):=v(rx)$ for $x\in \Omega'/r$ satisfies $Q_r[v_r]=0$ in $\tilde{\mathcal{A}}$. In light of  estimates \eqref{eq_AR} and \eqref{eq_AR1}, the norms $\|V_r\|_{M^q(p;\tilde{\mathcal{A}})}$ of the scaled potentials are uniformly bounded $\tilde{\mathcal{A}}$. Also, by \eqref{ell},  the matrices $A_r(x)$ are uniformly bounded and uniformly elliptic in $\tilde{\mathcal{A}}$. Therefore, the local Harnack inequality (Theorem~\ref{harnack_ineq}) in the annular domain $\tilde{\mathcal{A}}$ implies 
	\vspace{-3mm}
	$$\mathbf{A}_r=\underset{x\in \mathcal{A}_r}{\sup} \frac{u(x)}{v(x)}=\underset{x\in \mathcal{A}}{\sup} \frac{u_r(x)}{v_r(x)}\leq C\underset{x\in \mathcal{A}}{\inf} \frac{u_r(x)}{v_r(x)}=C\underset{x\in \mathcal{A}_r}{\inf} \frac{u(x)}{v(x)}=C \mathbf{a}_r,$$
	where the constant $C$ is independent of $u$ and $v$ for $r$ near $\gz$. 
\end{proof}
The weak comparison principle (Theorem~\ref{weak_comparison}) implies the following monotonicity useful result:
\begin{lem}\label{lemma_1}
	Let  $A$  and $V$ satisfy Assumptions~\ref{assump},  and assume that $u, v\in \mathcal{G}_\zeta$ are defined in a punctured neighbourhood $\Omega'$ of $\zeta$. For $r>0$, denote
	\be\label{eq_mM}
	m_r:=\underset{S_r\cap \Omega'}{\inf}\,\frac{u(x)}{v(x)}\,,\hspace*{0.5cm}M_r:=\underset{S_r\cap \Omega'}{\sup}\,\frac{u(x)}{v(x)}\,.
	\ee
	\begin{itemize}
		\item[(i)] The functions $m_r$ and $M_r$ are finally monotone as $r\rightarrow\zeta$. Specifically, there are numbers $0\leq m\leq M\leq \infty$ such that 
		\be\label{eq_lim_mM}
		m:=\underset{r\rightarrow\zeta}{\lim}\, m_r, \quad \mbox{and} \quad M:=\underset{r\rightarrow\zeta}{\lim}\, M_r.
		\ee
		\item[(ii)] Suppose further that $u$ and $v$ are both positive solutions of \eqref{p_laplace equ_2} in $\Omega$ of minimal growth in $\partial \hat{\Gw}\setminus\{\zeta\}$, then $0<m\leq M<\infty$ and $m_r\searrow m$ and $M_r\nearrow M$ when $r\rightarrow\zeta$.
	\end{itemize}
\end{lem}
The proof of Lemma~\ref{lemma_1} follows the  same steps as in  \cite[Lemma~4.2]{frass_pinchover} (where $A$ is the identity matrix and $V\in L^\infty_{\loc}(\Omega)$), and therefore it is omitted. 
\begin{remark}\label{mono_remrk}
	\em{
	Let $\mathbb A\!\in \!\mathbb{R}^{d\times d}$ be a symmetric, positive definite matrix. Then clearly, Lemma~\ref{lemma_1} also holds if the sphere $S_r$ is replaced by the set 
	$\partial E_{\mathbb A}(r)=\{x\in\mathbb{R}^d \mid |x|_{\mathbb A^{-1}}=r\}$. 
	 }
\end{remark}
The following result readily follows from the second part of Lemma~\ref{lemma_1}.
\begin{cor}\label{cor_1}
	Suppose that \eqref{p_laplace equ_2} has a Fuchsian isolated singularity at $\zeta \in\partial\hat{\Omega}$. Let $u, v$ be two positive solutions of \eqref{p_laplace equ_2} of minimal growth in a neighbourhood of $\partial\hat{\Omega}\setminus\{\zeta\}$. Then
	$$mv(x)\leq u(x)\leq M v(x)\quad x\in\Omega,$$
	where $0<m\leq M<\infty$ are defined in \eqref{eq_lim_mM}.
\end{cor}
As in \cite{frass_pinchover}, the regularity at $\gz$ implies a positive Liouville-type theorem.
\begin{prop}\label{minimum_growth}
	Suppose that $Q_{p,A,V}$ has a regular and isolated  Fuchsian singularity at $\zeta\in\partial\hat{\Omega}$. Then equation \eqref{p_laplace equ_2} admits a unique positive solution in $\Omega$ of minimal growth in a neighbourhood of $\partial\hat{\Omega}\setminus\{\zeta\}$.
\end{prop}
\begin{proof}
{\bf Existence:}  Follows from Proposition~\ref{minimum_growth1}.

{\bf Uniqueness:}	Let $u$ and $v$ be two solutions of  \eqref{p_laplace equ_2} of minimal growth in a neighbourhood of $\partial\hat{\Omega}\setminus\{\zeta\}$. Then by Corollary \ref{cor_1}, we have 
	$$ mv(x)\leq u(x)\leq M v(x)\quad x\in\Omega,$$
	where $0<m\leq M<\infty$ are defined in \eqref{eq_mM} and \eqref{eq_lim_mM}.
	In addition, since $\zeta$ is a regular point, it follows that
	\vspace{-3mm}  
	$$\underset{x\in\Omega}{\displaystyle{\lim_{x\rightarrow \zeta}}} \, \frac{u(x)}{v(x)}\quad\mbox{ exists and is positive}.$$
	Thus, we have $m=M$ and $u(x)=Mv(x)$.
\end{proof}
The following proposition asserts that the regularity of a Fuchsian singular point with respect to a limiting dilated equation implies the regularity of the corresponding singular point for the original equation \eqref{p_laplace equ_2}. The proposition  extends Proposition~2.2 in \cite{frass_pinchover}, where $A$ is the identity matrix and $V\!\in \! L^\infty_{\loc}(\Omega)$ satisfies \eqref{eq_pFs}. As in \cite{frass_pinchover}, the proof of below  relies upon  the Harnack convergence principle,  the WCP and the uniform Harnack inequality.
\begin{prop}\label{regularity_prop}
	Let   $A, V,$ satisfy Assumptions~\ref{assump}. Suppose that the operator $Q=Q_{p,A,V}$ has an isolated Fuchsian singularity at $\zeta\in\partial\hat{\Omega}$, and there is a sequence $R_n\rightarrow \zeta$, such that either $0$ or $\infty$ is a regular point of a limiting dilated equation $\mathcal{D}^{\{R_n\}}(Q)(w)=0$ in $\Gw$. Then $\zeta$ is a regular point of the equation $Q(u)=0$ in $\Omega$.
\end{prop}
\begin{proof} 
	Let $u, v \in \mathcal{G}_\zeta$ and  set 
	\vspace{-3mm}  
	\begin{equation}\label{eqn_1}
	m_r:= \underset{S_r\cap\Omega'}{\inf} \frac{u(x)}{v(x)},\quad M_r:= \underset{S_r\cap\Omega'}{\sup} \frac{u(x)}{v(x)},
	\end{equation}
	where $\Omega'$ is a punctured neighbourhood of $\zeta$. By Lemma~\ref{lemma_1}, $M:=\lim_{r\rightarrow\zeta}{M_r}$ and  $m:=\lim_{r\rightarrow\zeta}{m_r}$ exist in the wide sense, and we need to prove that $M=m$. 
	
	Now if $M:=\lim_{r\rightarrow\zeta}{M_r}=\infty$ (respectively, $m:=\lim_{r\rightarrow\zeta}{m_r}=0$), then by the uniform Harnack inequality, Lemma~\ref{uniform_harnack}, we have $m=\infty$ (respectively, $M=0$), and hence the limit 
	$$\underset{\substack{x\rightarrow\zeta\\x\in\Omega'}}{\lim}\frac{u(x)}{v(x)}\quad \text{exists in the wide sense.}$$
	Thus, we may assume that $u\asymp v$ in some neighbourhood $\Omega'\subset \Omega$ of $\zeta$. Fix $x_0\in \mathbb{R}^d$ such that $R_nx_0\in \Omega$ for all $n\in \mathbb{N}$ and define
	\vspace{-3mm}   
	\begin{equation}\label{eqn_2}
	u_n(x):= \frac{u(R_nx)}{u(R_nx_0)},\quad v_n(x):=\frac{v(R_nx)}{u(R_nx_0)}.
	\end{equation}
	Then by the definition of the set $\mathcal{G}_\zeta$, $u_n$ and $v_n$ are positive solutions of the equation
	\begin{equation}\label{eqn_3}
	-\Delta_{p,A_{n}}(w)+ V_n(x)|w|^{p-2}w=0\quad \mbox{in } \Omega'/{R_n},
	\end{equation}
	where $A_{n}(x):=A_{R_n}(x)$ and $V_n(x):=V_{R_n}(x)$, are the associated scaled matrix and potential, respectively. Since $u_n(x_0)=1$ and $v_n(x_0)\asymp 1$, the Harnack convergence principle (Proposition~\ref{covergence_principle}) implies that $\{R_n\}$ admits a subsequence (still denoted by $\{R_n\}$) such that  
	\begin{equation}
	\lim_{n\rightarrow\infty} u_n(x):=u_\infty(x), \mbox{ and } \lim_{n\rightarrow\infty}v_n(x):=v_\infty(x)
	\end{equation}
	locally uniformly in $Y =\displaystyle{\lim_{n\rightarrow\infty} \Omega'/R_{n}}$, and $u_\infty$ and $v_\infty$ are positive solutions of the limiting dilated equation 
	\begin{equation}\label{dilated_eq}
	\mathcal{D}^{\{R_n\}}(Q)(w) = -\Delta_{p,\mathbb{A}}(w)+\mathbb{V}|w|^{p-2}w=0\,\,\text{on}\,\,Y.
	\end{equation}
	Thus,  for any fixed $R>0$ we have 
	\begin{align*}
	\underset{x\in S_R}{\sup} \frac{u_\infty(x)}{v_\infty(x)} & =\underset{x\in S_R}{\sup}\lim_{n\rightarrow\infty}\frac{u_n(x)}{v_n(x)} =
	\lim_{n\rightarrow\infty}\underset{x\in S_R}{\sup}\frac{u_n(x)}{v_n(x)}\\[2mm]
	&=\lim_{n\rightarrow\infty}\underset{x\in S_R}{\sup}\frac{u(R_nx)}{v(R_nx)}=\lim_{n\rightarrow\infty}\underset{R_nx\in S_{RR_n}}{\sup}\frac{u(R_nx)}{v(R_nx)} = \lim_{n\rightarrow\infty} M_{RR_n} =M, 
	\end{align*}
	where we have used the existence of $\lim_{r\rightarrow\zeta}{M_r}=M$, and the local uniform convergence of the sequence $\{u_n/v_n\}$ in $Y$. Similarly, we have 
	$\underset{x\in S_R}{\inf} \frac{u_\infty(x)}{v_\infty(x)}=m$.
	Now by our assumption, either $\zeta_1=0$ or $\zeta_1 =\infty$ is a regular point of the dilated equation \eqref{dilated_eq}, so the limit 
	$$\underset{\substack{x\rightarrow\zeta_1\\x\in Y}}{\lim}\frac{u_\infty(x)}{v_\infty(x)}\quad \text{exists.}$$
	Hence, $m=M$, which implies that
	\vspace{-3mm}   
	$$\underset{\substack{x\rightarrow\zeta\\x\in \Omega'}}{\lim}\frac{u(x)}{v(x)}\quad \text{exists.}$$
	Thus,  $\zeta$ is a regular point of the equation $Q(u)=0$ in $\Omega$. 
\end{proof}
\section{Asymptotic behaviour of $(p,\mathbb{A})$-harmonic functions}\label{sec_pA_harmonic}
In this section, we study the regularity of positive $(p,\mathbb{A})$-harmonic functions at $\gz$, when $\mathbb{A}\in \R^{d\times d}$ is a fixed symmetric and positive definite matrix. We prove the following theorem.
\begin{theo}\label{regularpoint_lap}
Assume that $1<p<\infty$ and $d\geq 2$. Let $\mathbb{A}\in \R^{d\times d}$ be a fixed symmetric and positive definite matrix. Then
\begin{itemize}
	\item[(i)] for $p\leq d$, $\zeta =0$ is a regular point of the equation $-\Delta_{p,\mathbb{A}} (u)=0$ in $\mathbb{R}^d\setminus{\{0\}}$.
	\item[(ii)] for $p\geq d$, $\zeta=\infty$ is a regular point of the equation $-\Delta_{p,\mathbb{A}} (u)=0$ in $\mathbb{R}^d$.
\end{itemize}       
\end{theo}

\medskip
The proof of Theorem \ref{regularpoint_lap} depends on the asymptotic behaviour of positive solutions near an isolated singularity. Before proving Theorem~\ref{regularpoint_lap}, we first establish  the existence of a `fundamental solution' (given by an explicit form) for  the $(p, \mathbb{A})$-Laplacian 
$$-\Delta_{p,\mathbb{A}} (u)=\text{div}(|\nabla u|_\mathbb{A}^{p-2}\mathbb{A}\nabla u),$$ 
where $\mathbb{A}\in \R^{d\times d}$ is a fixed symmetric and positive definite matrix.
\begin{lemma}[Fundamental solution]\label{lem-fs}
	 Let $\mathbb{A}\in \R^{d\times d}$ be a fixed symmetric, positive definite matrix, and let $\mathbb{A}^{-1}$ be its inverse matrix. Fix  $y\in \R^d$. Let  
	 \begin{equation}\label{fundamental_soln}
	 \gm(x-y):= C_{p,d,\mathbb{A}}
	 \begin{cases} 
	 |x-y|_{\mathbb{A}^{-1}}^{(p-d)/(p-1)}& x\in \R^d, p\neq d,\\[2mm]
	 -\log |x-y|_{\mathbb{A}^{-1}} & x\in \R^d, p=d,
	 \end{cases}
	 \end{equation}
	 where
	 \begin{equation*}
	 C_{p,d,\mathbb{A}}:=
	 \begin{cases} 
	 \dfrac{p-1}{d-p}(|\mathbb{A}|^{1/2}\gw_d)^{-1/(p-1)}& p\neq d,\\[4mm]
	 (|\mathbb{A}|^{1/2}\gw_d)^{-1/(d-1)} & p=d,
	 \end{cases}
	 \end{equation*}
	 $|\mathbb{A}|$ is the determinant of  $\mathbb{A}$, and $\gw_d$ is the hypersurface area of the unit sphere in $\R^d$.	 
Then   
$$-\Delta_{p,\mathbb{A}} (\gm(x-y)) =\gd_y(x)   \qquad \mbox{ in } \R^d.$$
\end{lemma}
\begin{remarks}\label{rem_abuse}
{\em 1. Note that $\gm$ is a positive function if and only if $p<d$, which implies that $-\Delta_{p,\mathbb{A}}$ (where $\mathbb{A}$ is a constant matrix) is subcritical in $\R^d$ if and only if $p<d$ (see Theorem~\ref{criticality}).
	
	2. In the sequel  we abuse the notation and write $\gm(|x-y|):=\gm(x-y)$.}
\end{remarks}
\begin{proof}
Denote $C:= C_{p,d, \mathbb{A}}$, and assume first that $p\neq d$.  Without loss of generality, we may assume that $y=0$. Recall that for a fixed symmetric  matrix $\mathbb{A}\in \R^{d\times d}$, the gradient of the associated quadratic form is given by 
\vspace{-3mm}
\begin{equation}\label{grad_qf}
\nabla (\mathbb{A}x\cdot x)= 2\mathbb{A}x.
\end{equation}
	 Therefore, the  chain rule and \eqref{grad_qf} implies 
 \begin{align*}
	     \nabla\! \gm(x)& \!=\!C\frac{p\!-\!d}{p\!-\!1} |x|_{\mathbb{A}^{-1}}^{(1-d)/\!(p\!-\!1\!)} \nabla\!\left(\!(\mathbb{A}^{-1}x\!\cdot\! x)^{1/2}\right) 
	      \!=\! C\frac{p\!-\!d}{p\!-\!1} |x|_{\mathbb{A}^{-1}}^{(1-d)/(p-1)}\!\frac{1}{2|x|_{\mathbb{A}^{-1}}}\! \nabla\! (\mathbb{A}^{-1}x\!\cdot\! x)\\[2mm]
	     & =C\frac{p-d}{p-1}|x|_{\mathbb{A}^{-1}}^{(1-d)/(p-1)}\frac{1}{|x|_{\mathbb{A}^{-1}}} \mathbb{A}^{-1}x
	     = C\frac{p-d}{p-1}|x|_{\mathbb{A}^{-1}}^{(1-d)/(p-1)-1}\mathbb{A}^{-1}x.
	\end{align*}
So, 
\vspace{-3mm}
\begin{align*}
|\nabla \gm(x)|_\mathbb{A}&\!=\! |C|\!\left|\frac{p-d}{p-1}\right|\!|x|_{\mathbb{A}^{-1}}^{((1-d)/(p-1)-1)}|\mathbb{A}^{-1}x|_\mathbb{A}\!=\!
= |C|\left|\frac{p-d}{p-1}\right|  |x|_{\mathbb{A}^{-1}}^{(1-d)/(p-1)}.
\end{align*}
Consequently,
\vspace{-3mm} 
\begin{align}\label{eta}
\gh(x)\!:=\! |\nabla \gm(x)|_\mathbb{A}^{p\!-\!2}\!\mathbb{A}\!\nabla \gm(x) \!=\! C|C|^{p-2}\! c(p,d)|x|_{\mathbb{A}^{-1}}^{\frac{(1-d)(p-2)}{p-1}} |x|_{\mathbb{A}^{-1}}^{\frac{1-d}{p-1}-1} \!x
 \!=\! C|C|^{p\!-\!2}\! c(p,d) |x|^{-d}_{\mathbb{A}^{-1}} x, 
\end{align}
where  $c(p,d) = \left|\frac{p-d}{p-1}\right|^{p-2}\frac{p-d}{p-1}$\,.


Denote $\eta_i(x): =C|C|^{p-2} c(p,d)|x|^{-d}_{\mathbb{A}^{-1}}x_i $. Then by \eqref{grad_qf} 
\begin{align*}
\frac{\partial \eta_i(x)}{\partial x_i} = 
C|C|^{p-2}\!c(p,d)\left( |x|_{\mathbb{A}^{-1}}^{-d}  -  d |x|^{-d-2}_{\mathbb{A}^{-1}}(\mathbb{A}^{-1}x)_i x_i\right).
\end{align*}
Therefore, for all $x \in \R^d\setminus \{0\}$ we have
\vspace{-3mm}
\begin{align*}
-\Delta_{p,\mathbb{A}} (\gm(x)) \!=\!-\diver \eta(x) = -C|C|^{p-2}c(p,d) \Big(\! d |x|_{\mathbb{A}^{-1}}^{-d} -
d |x|^{-d-2}_{\mathbb{A}^{-1}}\big(\sum_{i=1}^d(\mathbb{A}^{-1}x)_i x_i \big)\!\Big) \!=\!0 .
\end{align*}
Similarly, for $p=d$, we obtain that $C_{d,\mathbb{A}}\log(|x|_{\mathbb{A}^{-1}})$ satisfies $-\Delta_{d,\mathbb{A}}(u)=0$ in $\mathbb{R}^d\setminus\{0\}$. 

We now find the constant $C= C_{p,d, \mathbb{A}}\in \R$ such that $\mu$ satisfies 
\begin{equation}\label{fund_eqn}
-\Delta_{p,\mathbb{A}}(\gm) =\delta_0 \qquad \mbox{in }\R^d,
\end{equation}
in the sense of distributions. Recall that the ellipsoid
$E_\mathbb{A}(r)=\{x\in \mathbb{R}^d \mid |x|_{\mathbb{A}^{-1}}< r\}$  
with `center' at the origin and `radius' $r>0$, is {\it affinely equivalent} to the ball $B_r(0)$. Hence, $E_\mathbb{A}(r)$ is a relatively compact, convex subset of $\mathbb{R}^d$.

Let us first consider the case $p\neq d$. Note that if $p<d$, then  $\lim_{x\to 0}\mu(x)=\infty$, but nevertheless,  $\gm$ is integrable near the origin. Using the divergence theorem on the ellipsoid $E_\mathbb{A}(r)$, it follows that the function $\gm$ should satisfy
\begin{align}
-1= \int_{E_\mathbb{A}(r)}\diver\!\! \left(|\nabla \mu|^{p-2}_\mathbb{A} \mathbb{A}\nabla\mu\right)\!\! \dx
 =\int_{\partial E_\mathbb{A}(r)}|\nabla \mu|^{p-2}_\mathbb{A}\mathbb{A}\nabla \mu\cdot {\bf{n}}\d S, \label{fund_eqn1}
\end{align}
where ${\bf n}=\mathbb{A}^{-1}x/|x|_{\mathbb{A}^{-1}(x)}$ is the unit outward normal to the boundary of the ellipsoid  $E_\mathbb{A}(r)$ and $\dS$ is the hypersurface element area. 
Recall that  by \eqref{eta} we have
$$|\nabla \mu(x)|^{p-2}_\mathbb{A} \mathbb{A}\nabla \mu(x)= C|C|^{p-2}\left|\frac{p-d}{p-1}\right|^{p-2}\frac{p-d}{p-1}\frac{x}{|x|_{\mathbb{A}^{-1}}^d}\,,$$ 
and the hypersurface area of $\partial E_\mathbb{A}(r)$ is given by 
$r^{d-1}|\mathbb{A}|^{1/2}\gw_d$ (see for example, \cite[p. 238]{manfred}), it follows  from \eqref{fund_eqn1} that
\begin{align*}
-1 & = C|C|^{p-2} \left|\frac{p-d}{p-1}\right|^{p-2}
\frac{p-d}{p-1} 
\int_{\partial E_\mathbb{A}(r)}\frac{1}{|x|^{d-1}_{\mathbb{A}^{-1}}}\dS
=  C|C|^{p-2} \left|\frac{p-d}{p-1}\right|^{p-2}\frac{p-d}{p-1} |\mathbb{A}|^{1/2}\gw_d.\label{fund_eqn2}
\end{align*}
So, for $p\neq d$, we have 
\vspace{-2mm}
$$C=C_{p,d,\mathbb{A}}= \frac{p-1}{d-p}(|\mathbb{A}|^{1/2}\gw_d)^{-1/(p-1)}.$$
Similarly, for $p=d$ one obtains $C(d, \mathbb{A})= (|\mathbb{A}|^{1/2}\gw_d)^{-1/(d-1)}$. 
\end{proof}
\begin{theorem}\label{assythm_1}
		Let $1<p\leq d$ and $\mathbb{A}\in \R^{d\times d}$ be a fixed symmetric positive definite matrix. Suppose that $u$ is a positive solution of the equation $-\Delta_{p,\mathbb{A}}(v)=0$ in a punctured neighbourhood of $0$ which has a non-removable singularity at $0$, then
		\vspace{-3mm} 
		\begin{equation*}\label{eqn_assy}
		u(x)\underset{x\rightarrow 0}{\sim} \gm(x),
		\end{equation*}
		where $\gm$ is the fundamental solution of $-\Delta_{p,\mathbb{A}}$  in $\R^d$ given by \eqref{fundamental_soln}.
\end{theorem}
\begin{remark}
	{\em For the case when $\mathbb{A}=I$, Theorem~\ref{assythm_1} has been proved in \cite[Theorem 1.1 and \cite{veron_E}]{veron}, see also \cite{gv}. 
		We give a slightly different proof of Theorem~\ref{assythm_1} by using Lemma~\ref{lemma_1}.}
\end{remark}
\begin{proof}[Proof of Theorem~\ref{assythm_1}] Assume that $1<p<d$, the proof for the case when $p = d$ needs only minor modifications, and therefore, it is omitted. 
	
	It is known \cite{serin_1, serin_2} that any positive solution $v$ of the equation $-\Delta_{p,\mathbb{A}}(u)=0$ in a punctured ball $B_r\setminus \{0\}$ has either a removable singularity at $0$, or else,  
	\begin{equation*}
	v(x)\asymp \mu(x) \qquad \mbox{as}\,\,x\rightarrow 0.
	\end{equation*}
	Since $u$ has a nonremovable singularity at $0$, it follows that there exists $C>0$ such that  $C^{-1}\mu(x)\leq v(x)\leq C\mu(x)$ for all $x$ in a small punctured neighbourhood of $0$.

	Let $\{x_n\}$ be a sequence converging to $0$. Denote $r_n=|x_n|_{\mathbb{A}^{-1}}$, and  define
	$$M_n:= \underset{\partial E_\mathbb{A}(r_n)}{\sup}\frac{u(x)}{\mu(x)}\,, \qquad m_n:=\underset{\partial E_\mathbb{A}(r_n)}{\inf}\frac{u(x)}{\mu(x)}\, .$$ 
	Then the sequence $\{M_n\}$ is bounded and bounded away from $0$. Moreover, by Lemma \ref{lemma_1} and Remark~\ref{mono_remrk}, the sequence $\{M_n\}$ is finally monotone. Let $M:=\lim_{n\rightarrow\infty}M_n$. Then 
	$$\lim_{n\rightarrow\infty}\big[\underset{\partial E_\mathbb{A}(r_n)}{\sup}(u/\mu-M)\big]=0.$$
	Fix $x_0\in \mathbb{R}^d$ such that $r_nx_0\in \Omega^*=\Omega\setminus\{0\}$ for all $n\in \mathbb{N}$ and consider the following functions
	$$u_n(x):= \frac{u(r_nx)}{\mu(r_nx_0)},\,\,\text{ and }\,\, \mu_n(x):= \frac{\mu(r_nx)}{\mu(r_nx_0)}.$$
	Then $u_n$ and $\mu_n$ are positive solution of the equation 
	$-\Delta_{p,\mathbb{A}}(w)=0$ in $\Omega^*/r_n$.
	Note that $\mu_n(x)=|x_0|_{\mathbb{A}^{-1}}^{\frac{d-p}{p-1}}|x|_{\mathbb{A}^{-1}}^{\frac{p-d}{p-1}}$ with $\mu_n(x_0)=1$, hence, $\mu_n$ does not depend on $n$.   On the other hand, $u_n(x_0)\asymp 1$, hence, the Harnack convergence principle implies that, up to a subsequence,
	$$\lim_{n\rightarrow \infty} u_n(x)=u_\infty(x)$$
	locally uniformly in $\mathbb{R}^d\setminus\{0\}$ and $u_\infty$ is a positive solution of 
	$-\Delta_{p,\mathbb{A}}(w)=0 \,\,\text{in}\,\,\mathbb{R}^d\setminus\{0\}$.
	Then for any fixed $R>0$, as in Proposition~\ref{regularity_prop}, it follows that
$$M=\underset{x\in \partial E_\mathbb{A}(R)}{\sup} \frac{u_\infty(x)}{\mu(x)}\,.$$
Hence, for any $R>0$, we have $u_\infty(x)\leq M \mu(x)$ for all $x\in \partial E_\mathbb{A}(R)$. Note that $\nabla \mu\neq 0$. Recall  the strong comparison principle, \cite[Theorem 3.2]{frass_pinchover} which is proved for the case where the principal part of the operator $Q$ is the $p$-Laplacian. Nevertheless, it is easy to check that the proof applies also to our setting, and in particular, for the $(p,\mathbb{A})$-operator. Hence, the  strong comparison principle implies that $u_\infty(x)= M \mu(x)$.
	
	
	Similarly, let $m:= \underset{n\rightarrow \infty}{\lim} m_n$, it follows that  for any $R>0$, we have $m\mu(x)\leq u_\infty(x)$ for all $x\in \partial E_\mathbb{A}(R)$, and consequently, $u_\infty(x)= m \mu(x)$. 
	Therefore,  $M=m$, and this implies that   
$$\lim_{n\rightarrow\infty}\left\|u/\mu-M\right\|_{L^\infty(\partial E_\mathbb{A}(|x_n|_{\mathbb{A}^{-1}})}=0.$$	
	In other words, $u$ is almost equal to $M\mu$ on a sequence of concentric ellipsoids converging
	to $0$. Using the WCP in the annuli 
	$$\mathcal{A}_n:=\{|x_{n+1}|_{\mathbb{A}^{-1}}\leq |x|_{\mathbb{A}^{-1}}\leq |x_{n}|_{\mathbb{A}^{-1}}\}, \quad n\geq 1,$$ 
	it follows that 
	\vspace{-3mm}
	$$\lim_{r\rightarrow 0}\|u/\mu-M\|_{L^\infty(\partial E_\mathbb{A}(r))}=0 .$$	
	Finally we note that it can be easily verified that $M$ is independent of the choice of the sequence $\{x_n\}$. Thus, the theorem is proved.		
\end{proof}
Similar to the case of the $p$-Laplacian in $\R^d$ we have:
\begin{theo}\label{criticality}
Assume that $\mathbb{A}\in\mathbb{R}^{d\times d}$ is a symmetric, positive definite matrix. Then the operator $-\Delta_{p,\mathbb{A}}$ is critical in $\R^d$ if and only if $p\geq d$.
\end{theo}
\begin{proof}
If $p<d$, then by the Hardy inequality for the $p$-Laplacian
	\begin{equation}\label{opt_hardy}
	\int_{\R^d}|\nabla \varphi|^p\dx \geq \left(\frac {d-p}{p}\right)^p\int_{\R^d} \frac{|\varphi|^p}{1+|x|^p}\dx \qquad \forall \varphi\in C_c^\infty(\R^d),
	\end{equation}	
	the operator  $-\Delta_{p,\mathbb{A}}$ is subcritical in $\R^d$.
	
	\medskip
	
	Suppose now that $p\geq d$. By Theorem~\ref{theorem:main}, the following Dirichlet problem admits a unique positive solution $w_k$:
	\begin{equation} \label{eq:minimal}
	\left\{
	\begin{array}{lr} 
	\Delta_{p,\mathbb{A}}(w_k)=0 & \qquad\mbox {in } B_k\setminus \overline{B_1}, \\
	w_k(x)=1 & \qquad \mbox{on } S_1, \\
	w_k(x)=0 & \qquad \mbox{on }  S_k.
	\end{array} \right.
	\end{equation}
	By the WCP, $\{w_k\}_{k\in\mathbb{N}}$ is an increasing sequence  satisfying $0\leq w_k\leq 1$, and therefore, converging to a positive solution $w$ of $\Delta_{p,\mathbb{A}}(v)=0$ in $\R^d\setminus \overline{B_1}$, that clearly has minimal growth at infinity in $\R^d$. Thus, it is enough to show that $w=1$ in $\R^d\setminus \overline{B_1}$. We obviously have  $w\leq 1$. On the other hand, since $|\gm(x)| \to \infty$ as $x\to\infty$, it follows that  for any $\varepsilon>0$, there is $k_{\varepsilon}$ such that $1- \varepsilon |\gm(x)|\leq w_k$ obviously on $S_1$ and also on $S_k$  for every $k\geq k_{\varepsilon}$. Invoking again  the WCP, it follows that $ 1-\varepsilon |\gm|\leq w$ in $B_k\setminus B_1$ and it follows $1-\varepsilon |\gm|\leq w$  in $\R^d\setminus B_1$. By letting $\varepsilon \to 0$, we conclude that $1\leq w$. Thus, $w=1$  in $\R^d\setminus B_1$, and $u_0=1$ is a ground state. Hence, by Theorem~\ref{theorem:main}, the operator $-\Delta_{p,\mathbb{A}}$ is critical in $\mathbb{R}^d$.
\end{proof}
\begin{cor}\label{assylem_2}
	Assume that $1<p<\infty$, and $\mathbb{A}\in\mathbb{R}^{d\times d}$ is a symmetric, positive definite matrix. Let $u$ be a positive solution of the equation $-\Delta_{p,\mathbb{A}}(u)=0$ in a neighbourhood of infinity. Then 	$\lim_{x\to \infty}u(x)$ exists in the wide sense.
	
	
	 Moreover, if $p\geq d$ (resp., $p<d$ ), then $\lim_{x\rightarrow \infty}u(x)\neq 0$ (resp,  $\lim_{x\rightarrow \infty}u(x)\neq \infty$). 
\end{cor}
\begin{proof}
	From Lemma~\ref{lemma_1} (with $v=1$), it follows that the functions given by
	$$m_r:=\inf_{x\in S_r} u(x),\quad M_r:=\sup_{x\in S_r}u(x)$$
	are monotone for large enough $r$. If ${\lim}_{r\rightarrow\infty}m_r=\infty$, then  clearly $\lim_{x\rightarrow \infty}u(x)=\infty$.
	
	
	Assume now that $m={\lim}_{r\rightarrow\infty}m_r<\infty$. Then for any $\epsilon>0$ the function $u-m+\epsilon$ is a positive solution of $\Delta_{p,\mathbb{A}}(w)=0$ in some neighbourhood infinity. Then by the uniform Harnack inequality \eqref{uniform_harnack}, we get
	\vspace{-3mm}
	$$M_r-m+\epsilon\leq C(m_r-m+\epsilon),$$
	for large enough $r$. By taking $r\rightarrow\infty$, we get $0\leq M-m\leq (C-1)\epsilon$. This implies that $M=m<\infty$, and $\underset{|x|\rightarrow\infty}{\lim}u(x)=m=M<\infty$. Thus, $u$ has a finite limit as $x\rightarrow\infty$.
	

	 Let $p<d$, and suppose that there exists a positive $(p,\mathbb{A})$-harmonic function $u$  in a neighbourhood such that $\lim_{x\rightarrow \infty}u(x) = \infty$. By repeating the proof of Theorem~\ref{criticality}, with $u$ replacing $\gm$, it would follow that $-\Delta_{p,\mathbb{A}}$  is critical in $\R^d$, a contradiction to Theorem~\ref{criticality}. 
	
	It remains to prove that  $\displaystyle{\lim_{x\rightarrow \infty}u(x)}\neq 0$ if $p\geq d$. By Theorem~\ref{criticality}, $\Delta_{p,\mathbb{A}}$ is critical in $\R^d$ with a ground state $u_0=1$.  Since a ground state is an entire positive of minimal growth at infinity, it follows that $\lim_{x\rightarrow \infty}u(x)\neq 0$.  Hence, the lemma follows.
\end{proof}
Next we discuss the asymptotic behaviour of positive $(p,\mathbb{A})$-harmonic functions at $\infty$ for $p\!\geq \! d$.  
\begin{theo}\label{assythm_2}
	Assume that $p\geq d$ and $\mathbb{A}\in\mathbb{R}^{d\times d}$ is a symmetric, positive definite matrix. Let $u$ be a positive solution of the equation $-\Delta_{p,\mathbb{A}}(w)=0$ in a neighbourhood of infinity in $\mathbb{R}^d$. Then either $u$ has a (finite) positive limit as $x\rightarrow\infty$, or
	\begin{equation*}
	u(x)\underset{x\rightarrow \infty}{\sim} -\gm(x),
	\end{equation*}
	where $\gm$ is the fundamental solution of $-\Delta_{p,\mathbb{A}}$  in $\R^d$ given by \eqref{fundamental_soln}.
\end{theo}
 To show this theorem, we use a Kelvin-type transform (see, Definition A.1 of \cite{frass_pinchover} for $\mathbb{A}=I$). 
 \begin{definition}{\em
 For $x\in \mathbb{R}^d$, we denote by $\tilde{x}:= x/|x|^2_{\mathbb{A}^{-1}}$. Then $\tilde{x}$ is the {\em inverse point with respect to the ellipsoid} $E_\mathbb{A}(1)$. In particular,  $|\tilde{x}|_{\mathbb{A}^{-1}}\!=\!1/|x|_{\mathbb{A}^{-1}}$.  Let $u$ be a function either defined in the ellipsoid $E_\mathbb{A}(1)\setminus\{0\}$, or on $\mathbb{R}^d\setminus E_\mathbb{A}(1)$. The {\em generalized Kelvin transform} of $u$ is given by 
 $$K[u](x):= u(\tilde{x})= u(x/|x|^2_{\mathbb{A}^{-1}}).$$
}
 \end{definition}
 For $p\!=\!d$, the Dirichlet integral $\int_\Omega|\nabla u|_\mathbb{A}^d \,\d x$ is conformally invariant since $\lambda_{\min}|\nabla u|^d\leq|\nabla u|_\mathbb{A}^d\leq \lambda_{\max} |\nabla u|^d $, where $\lambda_{\min}, \lambda_{\max}$ are the lowest and greatest eigenvalues of $\mathbb{A}$. The $(d,\mathbb{A})$-harmonic equation $-\Delta_{d,\mathbb{A}}(u)=0$ is therefore, invariant under the  generalized Kelvin transform. In particular, if $u$ is $(d,A)$-harmonic, then  $K[u]$ is also  $(d,\mathbb{A})$-harmonic (see also, Lemma~\ref{assylem_1}). Hence, for $p=d$,
  Theorem~\ref{assythm_2} follows from Theorem~\ref{assythm_1}.
\begin{lemma}\label{assylem_1}
	Assume that $p>d$, and let $\mathbb{A}\in\mathbb{R}^{d\times d}$ be a symmetric, positive definite matrix. Set $\beta:=2(p-d)$. Suppose that $u$ is a solution of $-\Delta_{p,\mathbb{A}}(u)=0$ in a neighbourhood of infinity (respectively, in a  punctured neighbourhood of origin). 
	
	Then  $v:= K[u]$ is a solution of the equation 
	\begin{equation}
	-\emph{div}\big(B(v)\big):=-\emph{div}\,(|x|^\beta_{\mathbb{A}^{-1}}|\nabla v|^{p-2}_\mathbb{A}\mathbb{A}\nabla v)=0,
	\end{equation}
	in a  punctured neighbourhood of origin (respectively, in a neighbourhood of infinity). 
\end{lemma}
\begin{proof}
	Denote $\tilde{x_i}:= x_i/|x|^2_{\mathbb{A}^{-1}}$. By using the chain rule and \eqref{grad_qf}, it follows that
	$$\nabla v(x)=|\tilde{x}|^2_{\mathbb{A}^{-1}}\tilde{\nabla}u(\tilde{x})-2(\tilde{\nabla}u(\tilde{x})\cdot\tilde{x}) \mathbb{A}^{-1}\tilde{x},$$
	where $\tilde{\nabla}$ denotes the gradient with respect to $\tilde{x}$. Therefore, 
	\begin{align*}
	|\nabla v(x)|_\mathbb{A}^2
	\!=\!\big[\mathbb{A}|\tilde{x}|^2_{\mathbb{A}^{-1}}\tilde{\nabla}u(\tilde{x}) \! - \! 2(\tilde{\nabla}u(\tilde{x}) \! \cdot \! \tilde{x}) \tilde{x}\big]\!\cdot \!
	\big[|\tilde{x}|^2_{\mathbb{A}^{-1}}\tilde{\nabla}u(\tilde{x}) \! - \! 2(\tilde{\nabla}u(\tilde{x}) \! \cdot \! \tilde{x}) \mathbb{A}^{-1}\tilde{x}\big]
	\!=\!|\tilde{\nabla}u(\tilde{x})|_\mathbb{A}^2|\tilde{x}|^4_{\mathbb{A}^{-1}}.
	\end{align*}
	Thus, $|\nabla v(x)|_\mathbb{A}=|\tilde{\nabla}u(\tilde{x})|_\mathbb{A}|\tilde{x}|^2_{\mathbb{A}^{-1}}$.
	
	Consider $B(v)= |x|^\beta_{\mathbb{A}^{-1}}|\nabla v|^{p-2}_\mathbb{A}\mathbb{A}\nabla v$, where $\beta=2(p-d)$. Following the same steps of the computation in \cite[Lemma A.1]{frass_pinchover}, we conclude that
	$$\mbox{div}\big(B(v)\big)=\mbox{div}(|x|^\beta_{\mathbb{A}^{-1}}|\nabla v|^{p-2}_\mathbb{A}\mathbb{A}\nabla v)=|\tilde{x}|^{2d}_{\mathbb{A}^{-1}}\tilde{\Delta}_{p,\mathbb{A}}(u(\tilde{x}))=0.
	\qquad \qedhere $$
\end{proof}
\begin{remark}
\em{By Lemma~\ref{lem-fs},  $|x|^{(p-d)/(p-1)}_{\mathbb{A}^{-1}}$ is a positive $(p,\mathbb{A})$-harmonic function in the punctured space. Lemma~\ref{assylem_1} implies that   $\mbox{div}\big(B(|x|^{(d-p)/(p-1)}_{\mathbb{A}^{-1}})\big)=0$ in the punctured space.
}
\end{remark}
The following two lemmas are the analogous results for the $p$-Laplacian proved in \cite[Appendix~A]{frass_pinchover}. For the completeness, we  provide the proof. 
\begin{lemma}\label{assylem_3}
	Assume that $p\!>\!d$, and $\mathbb{A}\! \in\! \mathbb{R}^{d\times d}$ is symmetric, positive definite matrix. Let $u$ be a solution of the equation $-\Delta_{p,\mathbb{A}}(u)=0$ in a neighbourhood of infinity  with $\lim_{x\rightarrow \infty}u(x)=\infty$. Choose $R>0$ and $c>0$ such that $v_c:=K[u](x)-c$ is positive near the origin and negative on $\partial E_\mathbb{A}(R)$. Then  there exists $C>0$ such that for any $\varphi\in C_0^1(E_\mathbb{A}(R))$ which equals $1$ near the origin, we have 
	\vspace{-3mm}
	$$\int_{E_\mathbb{A}(R)}B[v_c]\cdot\nabla\varphi \d x =C.$$
\end{lemma}
\begin{proof}
	The difference of any two such $\varphi$ has a compact support in $E_\mathbb{A}(R)\setminus\bar{E}_\mathbb{A}(0,\epsilon)$ for some $\epsilon>0$. Since $v_c$ satisfies $-\mbox{div}\,(B(v_c))=0$ in $E_\mathbb{A}(R)\setminus\bar{E}_\mathbb{A}(0,\epsilon)$, therefore it follows that 
	$$\int_{E_\mathbb{A}(R)}B[v_c]\cdot\nabla\varphi \d x =\mbox{ constant }=C.$$
	We show that the constant $C$ is positive. For this, we choose the following test function:
	$$
	\varphi_\nu(x):= 
	\begin{cases}
	0 & v_c(x)\leq 0,\\
	v_c(x) & 0<v_c(x)<\nu,\\
	\nu & v_c(x)\geq \nu.
	\end{cases} 
$$
	Therefore, we have 
	$$C=\int_{E_\mathbb{A}(R)}B[v_c]\cdot\nabla\varphi_1 \d x =\int_{\{x\in E_\mathbb{A}(R):\, 0<v_c<1\}}|x|^\beta_{\mathbb{A}^{-1}}|\nabla v_c|^p_\mathbb{A}\d x>0. \qquad\qedhere$$
\end{proof}

\begin{lemma}\label{assylem_4}
	Assume that $p>d$, and $\mathbb{A}\in\mathbb{R}^{d\times d}$ is a symmetric, positive definite matrix. Let $v_c(x)$ be the solution as in Lemma~\ref{assylem_3}. Then there exists $\epsilon>0$ such that 
	\begin{equation}\label{estimate_1}
		v_c\asymp |x|_{\mathbb{A}^{-1}}^{(d-p)/(p-1)}\quad \mbox{ in } E_\mathbb{A}(\epsilon)\setminus\{0\}.
	\end{equation}
\end{lemma}
\begin{proof}
	For $0<r<R$, consider
	$$m_r=\underset{x\in \partial E_\mathbb{A}(r)}{\inf} v_c(x)\quad\mbox{ and }\quad M_r= \underset{x\in \partial E_\mathbb{A}(r)}{\sup} v_c(x).$$
	 Since $\displaystyle{\lim_{x\rightarrow 0}v_c(x)=\infty}$, Remark~\ref{mono_remrk} (with $v=1$) implies that the functions $m_r, M_r$ are non-decreasing when $r\rightarrow 0$. We show that there exists constants $C_1$ and $C_2$ such that 
	$$m_r\leq C_1r^{(d-p)/(p-1)}\leq C_2 M_r \quad \mbox{ for all } 0<r<r_0,$$
	for some $r_0>0$. Then by applying the uniform Harnack inequality to the $(p,\mathbb{A})$-harmonic function $u$, the claim of the lemma will follow.
	
		Let $\varphi_\nu$ as defined above, and note that $\varphi_\nu(x)=\nu$ near origin. Thus, by Lemma~\ref{assylem_3}, we have 
	\begin{align*}
	Cm_r\!=\!\!\int_{E_\mathbb{A}(R)}\!\!B[v_c]\cdot\nabla\varphi_{m_r} \!\d x\!=\!C_1'\!\int_{E_\mathbb{A}(R)} \!\!|x|_{\mathbb{A}^{-1}}^\beta|\nabla\varphi_{m_r}|_\mathbb{A}^{p}\!\dx 
	 \!\geq C_1''\frac{\lambda_{\min}^p}{\lambda_{\max}^\beta} m_r^p\, \mathrm{cap}_{p,\beta}(B_{r,R}),
	\end{align*}
where $\lambda_{\min}$ and $\lambda_{\max}$ denote the lowest and greatest eigenvalue of the matrix $\mathbb{A}$ and  	
$\mbox{cap}_{p,\beta}(B_{r,R})$ is the weighted $p$-capacity of the ball $B_r$ in $B_R$ with respect to the measure $|x|^\beta$. Then  by \cite[Emaple 2.2]{Heinonen}, it follows that 
\vspace{-1mm}
$$\mbox{cap}_{p,\beta}(B_{r,R})=C'(r^{(p-d-\beta)/(p-1)}-R^{(p-d-\beta)/(p-1)})^{1-p}.$$
Since $(p-d-\beta)/(p-1)=(d-p)/(p-1)$, we have
$$Cm_r^{1-p}\geq C_1''' (r^{(d-p)/(p-1)}-R^{(d-p)/(p-1)})^{1-p},$$
which implies that 
$$m_r\leq C_1\big(r^{(d-p)/(p-1)}-R^{(d-p)/(p-1)}\big)\leq C_1r^{(d-p)/(p-1)}.$$
Next we show $r^{(d-p)/(p-1)}\leq C_2 M_r$, for some $C_2>0$. Denote $\alpha=(d-p)/(p-1)$. For $0<r<R$, consider the following test function  
$$\psi_r(x):=\begin{cases}
1 & |x|_{\mathbb{A}^{-1}}<r,\\
\frac{|x|^\alpha_{\mathbb{A}^{-1}}-R^\alpha}{r^\alpha-R^\alpha} & 1\leq |x|_{\mathbb{A}^{-1}}\leq R,\\
0 & |x|_{\mathbb{A}^{-1}}>R.
\end{cases}$$
By Lemma \ref{assylem_3} and using the H\"{o}lder inequality, we have 
\begin{align}
C\!=\!\!\int_{E_\mathbb{A}(R)} \!\!\!\nabla\psi_r \cdot B[v_c] \!\dx 
\!\leq\!\! \left(\!\int_{E_\mathbb{A}(R)\setminus E_\mathbb{A}(r)}\hspace{-1cm}|x|^\beta_{\mathbb{A}^{-1}}|\nabla\psi_r|_\mathbb{A}^p\!\dx\!\right)^{\!\!1/p}
\! \!\left(\!\int_{E_\mathbb{A}(R)\setminus E_\mathbb{A}(r)}\hspace{-1cm}|x|^\beta_{\mathbb{A}^{-1}}|\nabla v_c|_\mathbb{A}^p\!\dx\!\right)^{\!\!(p-1)/p}\!.
\label{estimate_3}
\end{align}
Now,
\begin{align*}
\int_{E_\mathbb{A}(R)\setminus E_\mathbb{A}(r)}\hspace{-1cm}|x|^\beta_{\mathbb{A}^{-1}}|\nabla\psi_r|_\mathbb{A}^p\d x&=\frac{C'}{(r^\alpha-R^\alpha)^p}\big(r^{(\alpha-1)p+\beta+d}-R^{(\alpha-1)p+\beta+d}\big)
 = \frac{C'}{(r^\alpha-R^\alpha)^{p-1}}\, ,
\end{align*}
where we used $(\alpha-1)p+\beta+d=\alpha$. Thus, for small $r$, we get 
\begin{equation}
\int_{E_\mathbb{A}(R)\setminus E_\mathbb{A}(r)}\hspace{-1cm}|x|^\beta_{\mathbb{A}^{-1}}|\nabla\psi_r|_\mathbb{A}^p\d x\leq C' r^{-\alpha(p-1)}.\label{estimate_4}
\end{equation}
For the second term of \eqref{estimate_3}, we note that $v_c=\psi_{M_r}$ in $\{0\leq v_c\leq M_r\}$ which is a subset of $E_\mathbb{A}(R)\setminus E_\mathbb{A}(r)$. Thus we have
\begin{align}
\int_{E_\mathbb{A}(R)\setminus E_\mathbb{A}(r)}\hspace{-1cm}|x|^\beta_{\mathbb{A}^{-1}}|\nabla v_c|_\mathbb{A}^p\d x&\leq \int_{\{0\leq v_c\leq M_r\}}|x|^\beta_{\mathbb{A}^{-1}}|\nabla v_c|_\mathbb{A}^p\dx
\leq \int_{E_\mathbb{A}(R)}B[v_c]\cdot\nabla\psi_{M_r} \dx
=CM_r.\label{estimate_5}
\end{align}
Therefore, from \eqref{estimate_3}, \eqref{estimate_4} and \eqref{estimate_5}, we get
$$C_2'\leq r^{\alpha(1-p)/p}M_r^{(p-1)/p},$$
which shows that $r^{(d-p)/(p-1)} =r^\alpha\leq C_2M_r$.
\end{proof}
\begin{proof}[Proof of Theorem~\ref{assythm_2}]
	Let $p>d$. In light of  Corollary~\ref{assylem_2}, we need only to consider the case $u(x)\to \infty$ as $x\to\infty$. Since Lemma~\ref{assylem_4} implies that $v(x):=K[u](x) \asymp |x|^{(d-p)/(p-1)}_{\mathbb{A}^{-1}}$ near the origin, we need to  show that in fact, 
	 $v(x):=K[u](x) \sim |x|^{(d-p)/(p-1)}_{\mathbb{A}^{-1}}$ as $x\to 0$. Then in light of Lemma~\ref{assylem_1},   $u(x)\sim |x|^{(p-d)/(p-1)}_{\mathbb{A}^{-1}}$ as $x\rightarrow\infty$. 
	 
	 \medskip
	
	We follow the proof of \cite[Theorem~2.3]{frass_pinchover}. For $0<\sigma <1$, define $w_\sigma(x):=  v(\sigma x)/\sigma^\alpha$ where $\alpha=(d-p)/(p-1)$. Since $v_c\asymp |x|^\alpha_{\mathbb{A}^{-1}}$ in $E_\mathbb{A}(\epsilon)\setminus\{0\}$, it follows that $w_\sigma(x)\asymp|x|^\alpha_{\mathbb{A}^{-1}}$ in $E_\mathbb{A}(\epsilon/\sigma)\setminus\{0\}$ for some $\epsilon>0$ and also the family $\{w_\sigma\}_{0<\sigma<1}$ is locally bounded. By extracting a subsequence $\sigma_n\rightarrow 0$, we have that the sequence $\{w_{\sigma_n}\}$ converges locally uniformly to $w(x)$ in $\mathbb{R}^d\setminus\{0\}$. Moreover, $w$ is a positive solution of the equation 
	$$-\mathrm{div}\,(B(u))=0\quad \mbox{ in } \,\mathbb{R}^d\setminus\{0\}.$$
	Then by Remark~\ref{mono_remrk}, we have
	$$m:=\lim_{r\rightarrow 0}m_r=\lim_{r\rightarrow 0}\underset{x\in \partial E_\mathbb{A}(r)}{\inf}\frac{v(x)}{r^\alpha}, \quad M:=\lim_{r\rightarrow 0}M_r=\lim_{r\rightarrow 0}\underset{x\in \partial E_\mathbb{A}(r)}{\sup}\frac{v(x)}{r^\alpha}\, .$$
	This implies that $m|x|^\alpha_{\mathbb{A}^{-1}}\leq w(x)\leq M|x|^\alpha_{\mathbb{A}^{-1}}$. 
	Indeed,  for any $R>0$ we have 
	\begin{align*}
	\inf_{x\in \partial E_\mathbb{A}(R)} \frac{w(x)}{|x|^\alpha_{\mathbb{A}^{-1}}}&\!=\!\inf_{x\in \partial E_\mathbb{A}(R)}\lim_{n\rightarrow \infty} \frac{w_{\sigma_n}(x)}{|x|^\alpha_{\mathbb{A}^{-1}}}=\lim_{n\rightarrow \infty}\inf_{x\in \partial E_\mathbb{A}(R)} \frac{w_{\sigma_n}(x)}{|x|^\alpha_{\mathbb{A}^{-1}}}\\[1mm]
	&\!=\!\lim_{n\rightarrow \infty}\inf_{x\in \partial E_\mathbb{A}(R)} \frac{v(\sigma_n x)}{|\sigma_n x|^\alpha_{\mathbb{A}^{-1}}}
	\!=\!\lim_{n\rightarrow \infty}\inf_{x\in \partial E_\mathbb{A}(\sigma_n R)} \frac{v(x)}{|x|^\alpha_{\mathbb{A}^{-1}}}\!=\! \lim_{n\rightarrow \infty} m_{\sigma_n R}\!=\!m,
	\end{align*}
	where we used the local uniform convergence of $\{w_{\sigma_n}(x)/|x|^\alpha_{\mathbb{A}^{-1}}\}$. Similarly, we have$$\underset{x\in E_\mathbb{A}(R)}{\sup}\frac{w(x)}{|x|^\alpha_{\mathbb{A}^{-1}}}=M\qquad  \forall R>0.$$
	Hence,
	\vspace{-3mm}
	$$m|x|^\alpha_{\mathbb{A}^{-1}}\leq w(x)\leq M|x|^\alpha_{\mathbb{A}^{-1}}\qquad \forall R>0.$$
	Note that $|x|^\alpha_{\mathbb{A}^{-1}}$ is a positive solution of $-\mbox{ div }(B(u))=0$ in $\mathbb{R}^d\setminus\{0\}$ and the function $|x|^\alpha_{\mathbb{A}^{-1}}$ does not have any critical point. Hence, by the strong comparison principle (see, \cite[Theorem 3.2]{frass_pinchover}) which is valid also for the $(p,\mathbb{A})$-operator, we obtain $m|x|^\alpha_{\mathbb{A}^{-1}}=w(x)=M|x|^\alpha_{\mathbb{A}^{-1}}$, and hence, $m=M$.
\end{proof}
\begin{proof}[Proof of Theorem~\ref{regularpoint_lap}]
The proof follows directly from theorems~\ref{assythm_1} and \ref{assythm_2}. 		
\end{proof}
\section{Weak Fuchsian singularity and positive Liouville theorems}\label{sec_weak_Fuchs}
In this section we introduce the notion of weak Fuchsian singularity, and prove Conjecture~\ref{main_conj} for $Q$ which has  weak Fuchsian singularity at $\gz$ (see, Theorem~\ref{liouville_thm}).
\begin{defin}{\em 
	Let  $A$  and $V$ satisfy Assumptions~\ref{assump}.  Assume that $Q$  has an isolated  Fuchsian singularity $\zeta\in\partial \hat{\Omega}$, where $\zeta =0$ or $\zeta=\infty$. The operator $Q=Q_{p,A,V}$ is said to have a {\em  weak Fuchsian singularity} at $\zeta$ if  there exist $m$ sequences $\{R_n^{(j)}\}_{n=1}^\infty \subset\mathbb{R}_+$, $1\leq j\leq m$, satisfying $R_n^{(j)}\rightarrow \zeta^{j}$, where $\zeta^{(1)}=\zeta$, and $\zeta^{(j)}=0$ or $\zeta^{(j)}=\infty$ for $2\leq j\leq m$, such that 
	\begin{equation}\label{wfs}
	\mathcal{D}^{\{R_n^{(m)}\}}\circ\cdots\circ\mathcal{D}^{\{R_n^{(1)}\}}(Q)(w)=-\Delta_{p,\mathbb{A}} (w)\qquad \text{on }Y,
	\end{equation}
	where $\mathbb{A}\in \R^{d\times d}$ is a symmetric, positive definite matrix, and $Y=\underset{n\rightarrow \infty}{\lim} \Omega/R_n^{(1)}$.
}
\end{defin}
\begin{remark}{\em 
Example~2.1 in \cite{frass_pinchover} demonstrates that $m$ in \eqref{wfs} might be greater than 1. Moreover, although in this example $V\not \in M^q(p;B_1\setminus \{0\})$, the corresponding operator has a weak Fuchsian singularity at $\gz=0$.   
}
\end{remark}  
The next example shows that if $\gz=0$ and $V\in M^q(p;\Omega)$ for some punctured neighborhood $\Gw$ of the origin, and $A$ is continuous at $0$, then $Q$ has weak Fuchsian singularity at $\gz=0$.  
\begin{example}{\em
Assume that  $A\!\in \! L_{\rm loc}^{\infty}(\Gw;\R^{d\times d})$ is continuous at the isolated singular point $\gz=0$. Let $V\in M^q_{\text{loc}}(p;\Omega)$ has a Fuchsian singularity at $0\in \partial\hat{\Omega}$. Further suppose that $V\!\in\! M^q(p;B_1\cap\Omega)$. Then for any smooth function $\varphi$ having compact support in $B_r\!\setminus\! \{0\}$ we have
		\begin{align}
		\left|\int_{\Omega/R} R^pV(Rx)\varphi(x)\d x\right|
		\leq R^{p-d}\int_{\Omega}|V(x)||\varphi(x/R)|\d x
		\leq R^{p-d}\|\varphi\|_{\infty}\int_{\Omega\cap B_{Rr}}|V(x)|\d x\label{ex_1} .
		\end{align}
Take $R>0$ small enough such that $\Omega\cap B_{Rr}\subset \Omega\cap B_1$. Then for $p<d$, \eqref{ex_1} implies 
\begin{align*}
	\left|\int_{\Omega/R} R^pV(Rx)\varphi(x)\d x\right|&\leq R^{p-d}(Rr)^{d/q'}\|\varphi\|_{\infty}\frac{1}{(Rr)^{d/q'}}\int_{\Omega\cap B_{Rr}}|V(x)|\d x\\
	& \leq \|\varphi\|_{\infty} r^{d/q'}  \|V\|_{M^q(p;\Omega\cap B_1)}R^{p-d/q} \underset{R\rightarrow 0}{\longrightarrow} 0,
\end{align*} 
while for $p>d$
\vspace{-3mm} 
\begin{align*}
\left|\int_{\Omega/R} R^pV(Rx)\varphi(x)\d x\right|\leq \|\varphi\|_{\infty}  \|V\|_{L^1(\Omega\cap B_1)} R^{p-d} \underset{R\rightarrow 0}{\longrightarrow} 0.
\end{align*}
Similarly, for $p=d$ it can be seen that 
\begin{align*}
\left|\int_{\Omega/R} R^pV(Rx)\varphi(x)\d x\right|\leq \|\varphi\|_{\infty}  \|V\|_{M^q(d;\Omega\cap B_1)} \frac{1}{\log^{q/d'}(1/Rr)} \underset{R\rightarrow 0}{\longrightarrow} 0.
\end{align*}
Therefore, the operator $Q_{p,A,V}$ has a weak Fuchsian singularity at $0$.
}	
\end{example}
\begin{theo}[Liouville theorem]\label{liouville_thm}
	Let  $A$  and $V$ satisfy Assumptions~\ref{assump}.  Suppose that $\zeta\in\partial  \hat{\Omega}$ is an isolated singular point. Assume that the operator $Q=Q_{p,A,V}$ has a weak Fuchsian singularity at $\zeta$. Then $\zeta$ is a regular point of Equation~\eqref{p_laplace equ_2}. 
	
	In other words, if $u$ and $v$ are two positive solutions of the equation $Q_{p,A,V}(w)=0$ in a punctured neighborhood of $\zeta$, then 
	\begin{itemize}
		\item[(i)] 
		$\;\;\displaystyle{\lim_{x\rightarrow\zeta}\frac{u(x)}{v(x)}}$  exists in the wide sense.
		
		\item[(ii)] the equation $Q_{p,A,V}(w)=0$ admits a unique positive solution in $\Omega$ of minimal growth in a neighbourhood of $\partial\hat{\Omega}\setminus\{\zeta\}$.
	\end{itemize}
\end{theo}
\begin{proof}
	By Proposition~\ref{minimum_growth}, we have (i) $\Rightarrow$ (ii). Thus, we only need to show that $\underset{x\rightarrow\zeta}{\lim}\,\frac{u(x)}{v(x)}$ exists in the wide sense. Since the operator $Q$ has a weak Fuchsian singularity at $\zeta$, we have  
	\begin{equation}\label{lveqn1}
	\mathcal{D}^{\{R_n^{(m)}\}}\circ\cdots\circ\mathcal{D}^{\{R_n^{(1)}\}}(Q)(w)=-\Delta_{p,\mathbb{A}} (w)=0\qquad \mbox{in }\,\mathbb{R}^d\setminus\{0\},
	\end{equation}
	where $\mathbb{A}\in \R^{d\times d}$ is a symmetric, positive definite matrix. Recall that by Theorem~\ref{regularpoint_lap}  either $0$ or $\infty$ is a regular point of $-\Delta_{p,\mathbb{A}}$. Therefore, Proposition~\ref{regularity_prop} and a reverse induction argument implies that $\zeta$ is a regular point of the equation $Q_{p,A,V}(w)=0$.
\end{proof}
\section{Positive Liouville theorem in the  elliptically symmetric case}\label{sec_spher}
This section is devoted to the proof of Conjecture~\ref{main_conj} in the elliptically symmetric case. 
\begin{definition}
{\em Let $\mathbb{A}\in \mathbb{R}^{d\times d}$ be a symmetric, positive definite matrix. 
We say that $f:\Gw \to \R$ is {\em elliptically symmetric with respect to $\mathbb{A}$} if 
$f(x)=\tilde f(|x|_{\mathbb{A}^{-1}})$ for all $x \in \Gw$, where $\tilde f:\R_+ \to \R$. In the sequel, with some abuse of notation, we omit the distinction between $f$ and $\tilde f$.}
\end{definition}
{\bf Throughout the present section we fix  $\mathbb{A}\in \mathbb{R}^{d\times d}$ and assume that the potential $V\in M^q_{\text{loc}}(p;\Omega)$ is elliptically symmetric with respect to $\mathbb{A}$ i.e., $V(x)=V(|x|_{\mathbb{A}^{-1}})$}. 

Denote $r=|x|_{\mathbb{A}^{-1}}$, and let us calculate
$\Delta_{p,\mathbb{A}}(f(r))=\div(|\nabla f(r)|_\mathbb{A}^{p-2}\mathbb{A}\nabla f(r))$.

Using \eqref{grad_qf}, we obtain
$$\nabla f(r)= f'(r)\frac{\mathbb{A}^{-1}x}{r}  \quad \mbox{ and } |\nabla f(r)|_\mathbb{A}= \frac{|f'(r)|}{r}|\mathbb{A}^{-1}x|_\mathbb{A}=|f'(r)|.$$
Consequently, 
$$\eta :=|\nabla f(r)|^{p-2}_\mathbb{A}\mathbb{A}\nabla f(r)=|f'(r)|^{p-2}f'(r)\frac{x}{r}, \quad \mbox{and }$$
\begin{align*}
\frac{\partial \eta_i}{\partial x_i}= 
\frac{|f'(r)|^{p-2}f'(r)}{r}+ \frac{x_i (\mathbb{A}^{-1}x)_i}{r}\bigg[-\frac{|f'(r)|^{p-2}f'(r)}{r^2}+(p-1)\frac{|f'(r)|^{p-2}f''(r)}{r}\bigg].
\end{align*}
Therefore, we get
\begin{equation}\label{eq_rad_laplac}
\Delta_{p,\mathbb{A}}(f(r)) = \sum_{i=1}^{d}\frac{\partial \eta_i}{\partial x_i}
=|f'(r)|^{p-2}\bigg[(p-1)f''(r)+\frac{d-1}{r}f'(r)\bigg], \quad \mbox{where } r=|x|_{\mathbb{A}^{-1}}.
\end{equation}
\begin{lem}\label{symlem1}
	Let $\mathbb{A}\in \mathbb{R}^{d\times d}$ be a symmetric, positive definite matrix. Assume that the domain $\Omega$ and the potential $V$ are elliptically symmetric with respect to $\mathbb{A}$ and the equation $Q_{p,\mathbb{A},V}(u)=0$ possess a positive solution. Further, suppose that the operator $Q_{p,\mathbb{A},V}$ has a Fuchsian isolated singularity at $\zeta\in\{0,\infty\}$. Then for any $u\in \mathcal{G}_\zeta$, there exists an elliptically symmetric (with respect to $\mathbb{A}$) solution  $\tilde{u}\in \mathcal{G}_\zeta$ such that $u\asymp\tilde{u}$.
\end{lem}
\begin{proof}
	 We consider the case $\zeta=0$, the case when $\zeta=\infty$, can be shown similarly. Fix  $R>0$ such that $u$ is defined in the punctured ellipsoid $E_{\mathbb{A}}(2R)\setminus\{0\}$. Then for $0<\gr< R$, consider the following Dirichlet problem 
	\begin{equation}\label{symeq1}
	\begin{cases}
	Q_{p,\mathbb{A},V}(w)=0 &\mbox{ in } E_{\mathbb{A}}(R)\setminus \bar{E}_{\mathbb{A}}(\gr),\\
	w(x)=m_R & x \in \partial E_{\mathbb{A}}(R)\\
	w(x)=m_\gr & x\in \partial E_{\mathbb{A}}(\gr),
	\end{cases}
	\end{equation}
	where $m_r=\inf_{x\in \partial E_{\mathbb{A}}}(r) u(x)$. By Lemma~\ref{lem_Dirichlet}, there exists a unique solution $u_{\gr,R}$ to the Dirichlet problem \eqref{symeq1}. Moreover, from the unique solvability of the one-dimensional Dirichlet problem it follows that  $u_{\gr,R}$ is elliptically symmetric with respect to $\mathbb{A}$. Moreover, by the uniform Harnack inequality (Theorem~\ref{uniform_harnack}) and the WCP we have 
	$$u_{\gr,R}\leq u\leq Cu_{\gr,R} \quad  \mbox{in } E_{\mathbb{A}}(R)\setminus E_{\mathbb{A}}(\gr),$$  where  $C>0$ is independent of $\gr$.  
	
	 Applying the Harnack converging principle,  it follows that there exists a sequence $\gr_n\rightarrow 0$ such that $u_{\gr_n}\rightarrow \tilde{u}$ locally uniformly in $E_{\mathbb{A}}(R)\setminus\{0\}$, where $\tilde{u}$ is an elliptically symmetric positive solution of the equation $Q_{p,\mathbb{A},V}(w)=0$ in $E_{\mathbb{A}}(R)\setminus\{0\}$.
\end{proof}
\begin{theorem}
	Let $\mathbb{A}\in \mathbb{R}^{d\times d}$ be a symmetric, positive definite matrix. Assume that the domain $\Omega$ and the potential $V$ are elliptically symmetric with respect to $\mathbb{A}$ and the corresponding equation \eqref{p_laplace equ_2} possess a positive solution. Further, suppose that the operator $Q_{p,\mathbb{A},V}$ has a Fuchsian isolated singularity at $\zeta\in\{0,\infty\}$. Then 
\begin{itemize}
	\item[(i)] $\zeta$ is a regular point of \eqref{p_laplace equ_2}.
	\item[(ii)] the equation $Q_{p,\mathbb{A},V}(w)=0$ possess a unique positive solution in $\Omega$ of minimal growth in a neighbourhood of $\partial\hat{\Omega}\setminus\{\zeta\}$.
\end{itemize}	
\end{theorem}
\begin{proof}
	(i)  Assume first that $u,v\in\mathcal{G}_\zeta$, where $u$ is elliptically symmetric with respect to $\mathbb{A}$. Since the operator $Q_{p,\mathbb{A},V}$ has a Fuchsian isolated singularity at $\zeta$, hence Lemma~\ref{lemma_1}, Proposition~\ref{regularity_prop}, and the uniform Harnack inequality Theorem~\ref{uniform_harnack}, imply that either  
	$$\underset{\substack{x\rightarrow\zeta\\x\in\Omega'}}{\lim}\frac{u(x)}{v(x)}\quad \text{exists, and equal either to 0 or } \infty,$$
	or else, 
	$u\asymp v$ in some punctured neighbourhood $\Omega'\subset \Omega$ of $\zeta$. For a sequence $\{R_n\}$ which converges to $\zeta$, define $u_n(x)$ and $v_n(x)$ as in the proof of Proposition~\ref{regularity_prop} (see, \eqref{eqn_2}). Then, $u_n$ and $v_n$ are positive solutions of \eqref{eqn_3}. Following the arguments as in  Proposition~\ref{regularity_prop}, it follows that up to a subsequence 
	$$\lim_{n\rightarrow \infty} u_n(x)= u_\infty(x), ~\lim_{n\rightarrow \infty} v_n(x)= v_\infty(x),$$
	locally uniformly in $\mathbb{R}^d\setminus\{0\}$, and $u_\infty$, $v_\infty$ are positive solutions of the limiting dilated equation
	$$-\Delta_{p,\mathbb{A}}(w)+\mathbb{V}|w|^{p-2}w=0\,\,\text{in}\,\mathbb{R}^d\setminus\{0\}.$$
	Note that the potential $\mathbb{V}$ and the solution $u_\infty$ are elliptically symmetric with respect to $\mathbb{A}$. Moreover, as in Proposition~\ref{regularity_prop}, for any fixed $R>0$, we have 
	$$\underset{x\in S_R}{\sup} \frac{u_\infty(x)}{v_\infty(x)}=M,\,\,\underset{x\in S_R}{\inf} \frac{u_\infty(x)}{v_\infty(x)}=m,$$
	where $M=\lim_{r\rightarrow\zeta} M_r$ and $m=\lim_{r\rightarrow\zeta} m_r$ and $m_r$, $M_r$  are defined as in Lemma~\ref{lemma_1}. Assume that the potential $\mathbb{V}$ is nonzero, otherwise it has a weak Fuchsian singularity at $\zeta$ and the theorem follows from Theorem~\ref{liouville_thm}.

	Let $S_{u_\infty}$ be the set of critical points of $u_\infty$. Then it is closed and elliptically symmetric. Now if $\zeta$ is an interior point of $\hat{S}_{u_\infty}$ then $|\nabla u_\infty|=0$ in some punctured neighbourhood $\Omega'$ of $\zeta$. This implies that $u_\infty$ is constant near $\zeta$ which contradicts our assumption  that $\mathbb{V}\neq 0$ near $\zeta$. Therefore, there exists an annular set $\mathcal{\tilde{A}}= E_{\mathbb{A}}(R)\setminus E_{\mathbb{A}}(r)$ close to $\zeta$ such that $S_{u_\infty}\cap \mathcal{\tilde{A}}=\emptyset$. Hence by the strong comparison principle (see \cite[Theorem 2]{frass_pinchover}), which is also valid for $Q_{p,\mathbb{A},V}$, we obtain $mv_\infty=u_\infty=Mv_\infty$ in $\mathcal{\tilde{A}}$. So, $m=M$, and the theorem follows.
	
	Assume now that $u,v\in \mathcal{G}_\zeta$. Then by Lemma~\ref{symlem1}, there exists a elliptically symmetric solution $\tilde{u}\in \mathcal{G}_\zeta$ such that $u\asymp \tilde{u}$. By the proof before if follows that $u \sim \tilde{u}$ and the limit 
		$$\underset{\substack{x\rightarrow\zeta\\x\in\Omega'}}{\lim}\frac{v(x)}{\tilde{u}(x)}\quad \text{exists in the wide sense,}\quad  \mbox{and } \underset{\substack{x\rightarrow\zeta\\x\in\Omega'}}{\lim}\frac{u(x)}{\tilde{u}(x)} =C>0,$$
		which shows that 
		\vspace{-3mm}
		$$\underset{\substack{x\rightarrow\zeta\\x\in\Omega'}}{\lim}\frac{u(x)}{v(x)}\quad \text{exists in the wide sense.}$$
	(ii) Follows from Proposition~\ref{minimum_growth}.
	\end{proof}
\begin{center}
	{\bf Acknowledgements}
\end{center}
{\small The  authors  acknowledge  the  support  of  the  Israel  Science Foundation (grant  637/19) founded by the Israel Academy of Sciences and Humanities.}

\end{document}